\documentclass[12pt,letterpaper]{amsart}
\usepackage{ABOWZ}

\begin{document}
\title{Compactly supported $\bA^1$-Euler characteristic and the Hochschild complex}

\author[Arcila-Maya]{Niny Arcila-Maya}
\address{Department of Mathematics, Duke University, Durham~NC, US}
\email{niny.arcilamaya@duke.edu, kirsten.wickelgren@duke.edu}

\author[Bethea]{Candace Bethea}
\address{Department of Mathematics, University of South Carolina, Columbia~SC, US}
\email{betheac15@gmail.com}

\author[Opie]{Morgan Opie}
\address{Department of Mathematics, University of California, Los Angeles~CA, US}
\email{mopie@math.ucla.edu}

\author[Wickelgren]{Kirsten Wickelgren}

\author[Zakharevich]{Inna Zakharevich}
\address{Department of Mathematics, Cornell California, Ithaca~NY, US}
\email{zakh@math.cornell.edu}

\begin{abstract}
We show the $\bA^1$-Euler characteristic of a smooth, projective scheme over a characteristic $0$ field is represented by its Hochschild complex together with a canonical bilinear form, and give an exposition of the compactly supported $\bA^1$-Euler characteristic $\chi^c_{\bA^1}: K_0(\Var_{k}) \to \GW(k)$ from the Grothendieck group of varieties to the Grothendieck--Witt group of bilinear forms. We also provide example computations.
\end{abstract}

\maketitle

\section{Introduction}
The Euler characteristic is one of the first combinatorial topological invariants.  Leonhard Euler originally introduced it for polyhedra, claiming that for any Euclidean polyhedron it is the case that for any polyhedron 
\[
\mathrm{\#\ vertices} - \mathrm{\#\ edges} + \mathrm{\#\ faces} = 2.
\]
Although this is true for any convex polyhedron, this is a surprisingly difficult fact to state formally and correctly, and depends intrinsically on what one means by ``polyhedron''. The Euler characteristic turns out to be a \emph{topological} invariant, which is most classically defined for a finite CW complex $X$ to be
\[
\chi(X) \defeq \sum_{i=0}^\infty (-1)^i \dim_\Q H_i(X;\Q).
\]
To make the Euler characteristic well-defined for non-compact sets it is necessary to replace homology with cohomology with compact support; with this definition it follows that for a closed subspace $Z$ of $X$ with open complement $U$,
\begin{equation} \label{eq:add}
  \chi(X) = \chi(Z) + \chi(U).
\end{equation}

If instead of general topological spaces we consider only varieties (or schemes) over $\C$, the Euler characteristic of a variety $X$ can be defined by the formula
\[
\chi(X) \defeq \chi(X(\CC)) = \sum_{i=0}^\infty (-1)^i \dim_\Q H^i_c(X(\C);\Q) \in \mathbb{Z}.
\] 

This invariant can be elevated by replacing $\Q$ with any field $k$, and retaining more cohomological information. Instead of taking dimensions, we can consider the vector spaces $H^i_c(X(\C);k)$ as elements of $K_0(k)$, the Grothendieck group of the ground field.  This gives the new definition
\[
\chi_k(X) \defeq  \sum_{i=0}^\infty (-1)^i [H^i_c(X(\C);k)] \in K_0(k).
\]
As $K_0(k) \cong \mathbb{Z}$, with the isomorphism given by the dimension, this
may not appear to be a useful observation.  

However, in the case when $X$ is a smooth, projective variety, $X(\CC)$ is a manifold and Poincar\'e duality states that there is a duality on the cohomology.  By keeping track of the duality, we obtain an enriched Euler characteristic: such an Euler characteristic takes values in bilinear forms, rather than vector spaces.  The form-valued Euler characteristic has beautiful applications in topology \cite{Rohlin52} \cite{Freedman82} \cite{Donaldson83}. Stabilizing, we get an invariant taking values in the Grothendieck--Witt group, enriching the Euler characteristic valued in $K_0(k) \cong \mathbb{Z}$. 

	\begin{definition}
Let $k$ be a field.  The \emph{Grothendieck--Witt group} $\GW(k)$ is the free abelian group generated by isomorphism classes of $k$-vector spaces equipped with a symmetric, nondegenerate bilinear form, under the relation that 
\[
[V,b] + [V',b'] = [V\oplus V', b\oplus b'].
\]
	\end{definition}

To make this approach work correctly in the context of $\AA^1$-homotopy theory,
it is necessary to shift perspective. We will work with coherent cohomology, and the relevant additional structure is given by Grothendieck--Serre duality. Before describing the cohomological approach, we begin with the abstract definition of the categorical Euler characteristic.  For the rest of this paper we restrict to fields of characteristic $0$.

	\begin{definition}\label{abstracteuler}
Let $X$ be a smooth, projective variety over a field $k$. Then the motivic suspension spectrum $\Sigma^\infty_T(X_+)$ is dualizable in the stable motivic homotopy category $SH(k)$ \cite[Theorem 5.22]{hoyois-equivariant} \cite[Appendix A]{Hu_Picard} \cite[Section 1 or ArXiv version 3 Section 1.1]{Levine-EC} \cite{riou2005dualite} \cite[Section 2]{voevodsky2003motivic}. Thus there exists an element 
\[
\chi^{\AA^1}(X) \defeq \Big( 1_k \rto \Sigma^\infty_T(X_+) \wedge_k D(\Sigma^\infty_T(X_+)) \rto  D(\Sigma^\infty_T(X_+)) \wedge_k \Sigma^\infty_T(X_+) \rto 1_k\Big)
\]
in $ \End_{SH(k)}(1_k),$ where $1_k$ denotes the motivic sphere spectrum, the first and last maps are coevaluation and evaluation, the middle is the symmetry swap isomorphism, and $D(-)$ denotes dualizing. See for example \cite{hoyois2015quadratic} or \cite{Levine-EC} for further discussion. 
	\end{definition} 

It is a beautiful theorem of F. Morel that $\End_{SH(k)}(1_k) \cong \GW(k)$, see for example \cite{Mor06} or \cite[Theorem 1.23, Corollary 1.24]{A1-alg-top}, giving the categorical Euler characteristic $ \chi^{\AA^1}(X)$ in $\GW(k)$.  So, one can hope to recover an enhancement of a cohomological construction by explicitly describing the map in Definition \ref{abstracteuler}. 

A description in terms of coherent duality was independently suggested by M.J.Hopkins, A. Raksit, and J.-P. Serre in oral communication. It was proven to equal  $ \chi^{\AA^1}(X)$ by M. Levine and A. Raksit \cite[Theorem 1.3]{levine_raksit}. This is the description that motivated the introduction above, and as we will use it to prove our main result, we give a more detailed account now. 

For any vector space $V$ over $k$, write $V[i]$ for the chain complex consisting
of $V$ concentrated in degree $-i$, with $0$'s elsewhere. Consider a smooth and proper variety $X$ of pure dimension $d$. 

	\begin{definition}\label{defHdg} 
Let
\begin{equation} \label{eq:Hdg}
\Hdg(X/k)\defeq \bigoplus_{i,j}^d H^i(X, \Omega^j_{X/k})[j-i] \in \Ch_k.
\end{equation}

This is equipped with a symmetric bilinear form via the cup product and the canonical trace map $\Tr:H^d(X;\Omega^d_{X/k}) \rto k$
\begin{equation}\label{Hodge_pairing_eqn}
\begin{tikzcd}
\Hy^{i}(X,\Omega^{j}_{X/k})\otimes \Hy^{d-i}(X,\Omega^{d-j}_{X/k}) \arrow[r,"\cup"] & \Hy^{d}(X,\Omega^{d}_{X/k}) \arrow[r,"\Tr"]  & k.
\end{tikzcd}
\end{equation}
This structure produces a bilinear form $\Tr$ on $\Hdg(X/k)$. \end{definition} See \cite[Section 8D]{levine_raksit} for a review of the formalism of chain complexes with bilinear forms. The bilinear form $\Tr$ in fact recovers $\chi^{\bA^1}$ as defined above. 

	\begin{theorem}[{\cite[Theorem 1.3]{levine_raksit}}]\label{koszulworks}
For $X$ smooth and projective over a field $k$ of characteristic not $2$, $$\chi^{\AA^1}(X/k) = (\Hdg(X/k), \Tr) \in \GW(k).$$ 
	\end{theorem}

We remind the reader that in this paper $k$ will be a field of characteristic $0$. Another description is as follows. \begin{definition}\label{koszuleuler}Let $p_X: X \to \Spec k$ denote the structure map of a smooth proper variety $X$. Let \begin{equation*} 
\T_X:= \oplus^0_{i=-d}(\wedge^{-i} T^*_X),
\end{equation*} 
with zero differential. This is the Kozsul complex of the tangent bundle $T_X$ with respect to the zero section. It has a natural non-degenerate bilinear form $$\beta_X: \T_X \otimes \T_X \to \T_X \to \wedge^{d} T^*_X [d] $$ given by composing the multiplication of forms with projection to the $-d$th term of $\T_X$.
	\end{definition}

Because $X$ is smooth and proper, Serre duality allows us to pushforward a non-degenerate bilinear form valued in $\wedge^{d} T^*_X [d]$ and obtain a non-degenerate bilinear form valued in $k$ \cite[p. 7 Ideal Theorem c]{HartshorneRD}, \cite[Theorem 4.2.9]{calmes2009tensor}: 
\begin{equation}\label{equ:p_*Beta_XKoszul} 
\RR{p_X}_*\beta_X: \RR{p_X}_* \T_X \otimes  \RR{p_X}_* \T_X \to \RR{p_X}_* (\T_X \otimes \T_X) \to  \RR{p_X}_*\wedge^{d} T^*_X [d] \stackrel{\Tr}{\to} k \end{equation}
There is more discussion of this in \cite[Section 2.2]{bachmann_wickelgren} and we will discuss a similar situation in Lemma \ref{lm:BXnondegenerate}.

	\begin{proposition}[{\cite[Proposition 2.4]{bachmann_wickelgren}}]\label{koszulagrees}
Let $X$ be smooth and proper over $k$. With notation as above,
$(\RR{p_X}_* \T_X, \RR{p_X}_* \beta_X) = (\Hdg(X/k), \operatorname{Tr})$ in 
$\GW(k).$ 
	\end{proposition}

Each isomorphism class in $\GW(k)$ can be represented by a perfect chain complex over $k$ with a nondegenerate symmetric bilinear form. This is a result of Hermitian K-theory; see \cite{schlichting10} and \cite{Schlichting17} for background on Hermitian K-theory. A representing chain complex with duality provides the opportunity to do further homotopy theory. The example we have in mind will be described below. As a candidate representative, however, the Koszul complex has the following property: it depends on a section. There are situations in which one naturally has a section, e.g. \cite{CubicSurface}, but the association of the Euler characteristic to a scheme is arguably not one of them: We chose the zero section above, but is this a good choice? Any section would in fact do \cite[Proposition 2.4]{bachmann_wickelgren}, and it may be desirable not to choose at all. The main result of this paper allows this by offering the Hochschild complex as an alternative.

There is a canonical pairing on the Hochschild homology $\HH$ , which appears in greater generality in work of Alonso-Jeremías-Lipman, \cite{atjll}. This pairing can be viewed as coming from the connection between Grothendieck duality and Hochschild homology, due to Avramov, Lipman, and Iyengar, among others. We give a construction specific to our case of interest using work of Neeman in Section \ref{Formalism}. We then show that this agrees with $(\Hdg(X/k), \Tr)$ and $(\RR{p_X}_* \T_X, \RR{p_X}_* \beta_X)$ using the Hochschild--Kostant--Rosenberg theorem for smooth and proper schemes $X$ over a field of characteristic $0$.

	\begin{thmintro}
Let $X$ be a smooth, projective scheme over a field of characteristic $0$. $\HH(X)$ together with a canonical pairing represents the categorical $\AA^1$-Euler characteristic for smooth projective $X$.
	\end{thmintro} 

This is Theorem \ref{thm:B_XisChiA1_Xsmpr} in Section \ref{Formalism}. The example alluded to above is the following enrichment of the $\AA^1$-Euler characteristic. In joint as well as independent work, J. Campbell and the fifth-named author have constructed a spectrum whose $\pi_0$ is the Grothendieck group of varieties $K_0(\Var_k)$. See \cite{campbell, Z-kth-ass}. Independently, O. R\"ondigs \cite{rondigs} constructs a K-theory of varieties spectrum whose $\pi_0$ receives a surjective map from $K_0(\Var_k)$. (The Grothendieck group of varieties is by definition the universal cut-and-paste invariant, and is described in more detail in for example \cite{bittner04}.)  Bittner's presentation promotes the  $\AA^1$-Euler characteristic of a smooth, projective variety to a compactly supported  $\AA^1$-Euler characteristic for any variety over $k$. In more detail:

	\begin{theorem}[{\cite[Theorem 3.1]{bittner04}}] \label{bittnerpres}
Let $k$ be a field of characteristic $0$. The group $K_0(\Var_k)$ is isomorphic to the free abelian group generated by smooth projective varieties modulo the relation that for any smooth closed subvariety $Y \subseteq X$,
\[
[\Bl_YX] - [E] = [X] - [Y],
\]
where $\Bl_YX$ is the blow-up of $X$ at $Y$, and $E$ is the exceptional divisor of the blowup.
	\end{theorem}
 
	\begin{definition} \label{def:chic}
Let $X$ be a smooth scheme over a field $k$ of characteristic $0$.  We define $\chi_c(X)$ as follows. For $\dim X = 0$ we define $\chi^{\AA^1}_c(X) = \chi^{\AA^1}(X)$.  For $\dim X > 0$, write $[X] = \sum_{i=1}^n \epsilon_i [X_i]$ in $K_0(\Var_k)$, where each $X_i$ is smooth projective and $\epsilon_i = \pm 1$. Then define
\[
\chi^{\AA^1}_c(X) \defeq \sum_{i=1}^n \epsilon_i \chi^{\AA^1}_c(X_i).
\]
	\end{definition}

	\begin{theorem}
$\chi_c^{\AA^1}$ is well-defined, and induces a homomorphism $K_0(\Var_k) \rto \GW(k)$.
	\end{theorem}

This follows from \cite[Theorem 5.2, Corollary 6.7]{rondigs}. We also provide an exposition in Theorem~\ref{chi_c_cut_and_paste_char_0}. 

	\begin{question} 
Does the compactly supported $\AA^1$-Euler characteristic $\chi_c^{\AA^1}$ lift to a map of spectra from a spectrum level version of $K_0(\Var_k)$ to Hermitian K-theory? What kind of information would such maps encode? 
	\end{question}

O. R\"ondigs gives an affirmative answer to the first question in \cite[Theorem 6.6]{rondigs}. He constructs his map using the categorical $\AA^1$-Euler characteristic and Waldhausen's model of $K$-theory. We suggest an alternate approach these questions, studying the association of the Hochschild complex to a smooth, projective scheme, especially as related to Bittner's presentation. Section \ref{Example} provides sample computations. 

\subsection*{Acknowledgements}
We wish to thank Tom Bachmann and Marc Hoyois for useful discussions.

We likewise express our gratitude to the Women in Topology III program and the Hausdorff Research Institute for Mathematics, as well as support from the National Science Foundation NSF-DMS 1901795, NSF-HRD 1500481 - AWM ADVANCE grant and Foundation Compositio Mathematica. Niny Arcila-Maya was supported by St. John's college Itoko Muraoka fellowship. Candace Bethea was supported by the Southern Region Education Board under the Dissertation Award. Morgan Opie was partially supported by National Science Foundation Award DGE-1144152.  Kirsten Wickelgren was partially supported by National Science Foundation Awards DMS-1406380 and CAREER-DMS-1552730.  Inna Zakharevich was partially supported by National Science Foundation Award CAREER-1846767. 

\subsection*{Notation}
$k$ will denote a field of characteristic 0.

By a \emph{variety over $k$} we mean a reduced, irreducible, separated schemes
of finite type over a field $k$. For a variety $X$ over $k$, the map $p_X$
denotes the structure map $p_X:X \rto \Spec k$.  When $X$ is clear from context
we omit it from the notation, and write $p$ for the structure map. Our chain
complexes are cohomological and often concentrated in negative degrees. Write
$\operatorname{hCoh}(X)$ for the derived category of coherent sheaves on $X$.

\section{Hochschild complex represents the $\bA^1$-Euler characteristic}
\label{Formalism}

We use the connection between Grothendieck duality and Hochschild homology to
give a bilinear form on the absolute Hochschild complex in this section. This form appears in greater generality in work of Alonso-Jeremías-Lipman \cite{atjll}. We then show this complex with duality represents the $\AA^1$-Euler characteristic for a smooth, projective scheme over a field of characteristic $0$, and use Bittner's presentation to give an exposition of the compactly supported $\AA^1$-Euler characteristic. Recall that we are working over a field of characteristic $0$.  

In future work we hope to use this formalism to lift the compactly supported Euler characteristic to a spectrum-level construction from J. Campbell and the fifth-named author's spectrum of varieties \cite{campbell} \cite{Z-kth-ass} landing in the Hermitian $K$-theory of $k$.  As working in the derived category is presumably not sufficient to make such a construction well-defined, we keep track (and include in our notation) which functors are defined outside the derived category, as well.  Thus, for example, given a map  $f:X \rto Y$ we think of $f^*$ as a functor $\Coh(Y) \rto \Coh(X)$ and denote  the derived functor by $\LL f^*: \operatorname{hCoh}(Y) \to \operatorname{hCoh}(X)$. 

Let $k$ be a field. We will shortly restrict to the case where $k$ is characteristic $0$.

	\begin{definition}
Let $p:X \to \Spec k$ be a smooth separated scheme over $k$, and let
$\Drcat(X)$ be the derived category of coherent sheaves on $X$.  Write
$\Delta: X \rto X\times X$ for the diagonal map, and define
$\Odelta_X \defeq \Delta_*\cO_X$.  This inherits a ring structure from
$\calO_X$; write $\mu_X: \Odelta_X \otimes \Odelta_X \rto \Odelta_X$ for the multiplication map.

Since $\Delta$ is affine, $\RR \Delta_* = \Delta_*$, so this is in fact a derived construction. Define the \emph{Hochschild complex $\HH^X(X)$} in $\Drcat(X)$ to be
\[
\HH^X(X) \defeq  (\LL\Delta^*) (\Delta_* \cO_X),
\]
which is equivalent to $\Odelta_X \otimes^{\LL}_{X\times X} \Odelta_X$. Since $\LL \Delta^*$ is strongly monoidal, the Hochschild complex inherits a multiplication 
\[
\LL \Delta^* \mu_X : \HH^X(X) \otimes \HH^X(X) \to \HH^X(X)
\]
  
We then define the \emph{absolute Hochschild complex $\HH(X)$} in $\Drcat(k)$ to be
\[
\HH(X) \defeq \RR p_*\HH^X(X).
\]
	\end{definition}

To see that the Hochschild homology is a good candidate to represent the $\AA^1$-Euler characteristic, we first note that it is a representative in $K_0(k)$ when the characteristic of $k$ is $0$:

	\begin{theorem}[Hochschild--Kostant--Rosenberg, {\cite{HKR62} \cite[p.1]{antieau_vezzosi}}]
Let $k$ be a field of characteristic $0$. For a smooth projective $X$, $[\HH(X)] = \Hdg(X/k)$ in $K_0(k)$.
	\end{theorem}

We would like the Euler characteristic to take values in $\GW(k)$, rather than $K_0(k)$. To accomplish this, we use the connection between Grothendieck duality and Hochschild homology found by Avramov and Iyengar \cite{avramov_iyengar}, and particularly the work of Neeman \cite{Neeman_GDHH} to construct a nondegenerate symmetric bilinear form on $\HH^X(X)$. This pairing appears in greater generality in Alonso-Jeremías-Lipman \cite{atjll}.

The multiplication on $\Odelta_X$ induces a multiplication
\begin{equation*}\label{def:mult}
\hat \mu_X:\HH(X)\otimes \HH(X) \rto \HH(X)
\end{equation*}
in the following manner. Define $\hat \mu_X$ to be the composition
\[
\HH(X) \otimes \HH(X) \rto \RR p_*(\HH^X(X)\otimes \HH^X(X)) \mathrel{\setlen{7em}{\rto^{(\RR p_*)(\LL\Delta_X^*)\mu_X}}} \HH(X),
\] 
where the first map is given by the lax monoidal structure on $\RR p_*$. Using this multiplication, $\HH(X)$ is equipped with the following canonical bilinear pairing.

	\begin{definition} \label{def:BX}
Let $X$ be proper, and let $p:X \rto \Spec k$ be the structure map.\footnote{Many of the constructions, and in particular the result of Neeman, are possible in much greater generality.}  Let $\pi_2: X\times X \rto X$ be the projection onto the second coordinate, as in the following commutative diagram:
\begin{diagram}
{ X \\
& X\times X & X \\ & X & \Spec k\\};
\to{1-1}{2-2}^\Delta \to{2-2}{3-2}_{\pi_2}
\to{2-2}{2-3}^{\pi_1}
\to{3-2}{3-3}^p \to{2-3}{3-3}^p
\end{diagram}
By \cite[Proposition 3.3]{Neeman_GDHH}, there is a canonical isomorphism 
\begin{equation*}
\label{Neeman_iso} \delta_X: (\LL\Delta^*) \pi_2^!\calO_X \to p^! \calO_k.
\end{equation*}
Inside the derived category, the adjoint to the weak equivalence $(\pi_2)_* \Odelta_X = (\pi_2)_* \Delta_* \cO_X \rwe (\pi_2\circ \Delta)_*\cO_X = \cO_X,$ is 
\[
t:\Odelta_X \to \pi_2^! \cO_X.
\] 

Define $\alpha'$ to be the composition of $\LL \Delta^*t$ with $\delta_X$ 
\[
\alpha': \HH^X(X) = (\LL \Delta^*) \Odelta_X \to (\LL \Delta^*) \pi_2^{!} \cO_X
  \rto^{\delta_X} p^!\cO_k.
\]
The map $\alpha'$ is a special case of the fundamental class of a flat map of Alonso-Jeremías-Lipman \cite[p. 390]{atjll}. Since $X$ is proper, there is a natural equivalence $p_! \rto p_*$; the map $\alpha$ adjoint to $\alpha'$ can thus be written as
\[
\alpha: \HH(X) \rto \cO_k.
\]
   
Define the bilinear form 
\[
B_X: \HH(X) \otimes \HH(X) \rto \calO_k \qquad\hbox{by} \qquad B_X \defeq \alpha \circ \hat \mu_X.
\]
	\end{definition}

Since $\mu_X$ is a commutative multiplication and $\RR p_*$ is lax symmetric monoidal, $\hat \mu_X$ is commutative; it follows that $B_X$ is symmetric.

	\begin{lemma}\label{lm:BXnondegenerate}
(Alonso-Jeremías-Lipman) $B_X$ is nondegenerate.
	\end{lemma}

Lemma~\ref{lm:BXnondegenerate} is a special case of Alonso, Jeremías, and Lipman's \cite[Corollary 4.8.4]{atjll}. We provide a proof adapted to our level of generality to provide a self-contained exposition.

	\begin{proof}
The map $\alpha' \circ (\LL \Delta_X^*)\mu_X$ determines a map
\begin{equation}\label{adjointHHtoHomHHp!}
\HH^X(X) \to \RR \Hom(\HH^X(X), p^! \calO_k),
\end{equation} 
which we claim is an isomorphism. To see this, observe that \eqref{adjointHHtoHomHHp!} is equal to the composition
\[
\LL \Delta^* \Odelta_X \rto \RR \Hom(\LL \Delta^* \Odelta_X, \LL \Delta^*\pi_2^! \calO_X) \rto^{\delta_X\circ } \RR \Hom(\LL \Delta^* \Odelta_X, p^! \cO_k),
\]
where the first map is induced from the multiplication. Since $\delta_X$ is an isomorphism, the second is as well, so it remains to show that 
\begin{equation}\label{torhomadjmultHH^X}
\LL \Delta^* \Odelta_X \rto \RR \Hom(\LL \Delta^* \Odelta_X, \LL \Delta^*
  \pi_2^! \calO_X)
\end{equation}
is an isomorphism. Since $X$ is smooth, $\Delta$ is perfect \cite[Section 37.54, Lemma 37.54.18.]{stacksproject}, and thus has finite Tor-dimension \cite[Section 0685 Lemma 37.53.11]{stacksproject}. Since $\Delta$ is proper and
finite Tor-dimension,
\[
\LL \Delta^* \RR \Hom(\Odelta_X,\pi_2^!\cO_X) \simeq \RR \Hom(\LL \Delta^* \Odelta_X, \pi_2^! \cO_X)
\]
by \cite[Lemma 4.3]{Neeman_GDHH}, and the morphism \eqref{torhomadjmultHH^X} is induced by $\LL \Delta^*$ applied to
\[
\Odelta_X \rto \RR \Hom(\Odelta_X, \pi_2^! \calO_X).
\]
This is an isomorphism, since it equal to the composition
\[
\Delta_* \cO_X \cong \Delta_* \RR\Hom(\cO_X,\cO_X) \cong \Delta_* \Hom(\cO_X,\Delta^! \pi_2^! \cO_X) \cong \Hom(\Delta_*\cO_X, \pi_2^! \cO_X),
\]
where the third isomorphism is by \cite[(2.5.2)]{Neeman_GDHH} and the fourth is
by \cite[p.14]{Neeman_GDHH}.
  
By coherent duality \cite[p. 7 Ideal Theorem c]{HartshorneRD}, the composition of the natural map 
\[
\RR \Hom_X(\HH (X), p^!\calO_k) \to \RR\Hom_Y(\RR p_* \HH(X), \RR p_* p^! \calO_k)
\] 
with the map induced by the trace $\RR p_* p^! \to 1$ is isomorphism 
\[
\RR p_*\RR \Hom_X(\HH(X),p^!\calO_k) \to^{\cong} \RR\Hom_k (\HH(X),\calO_k).
\] 
Since \eqref{adjointHHtoHomHHp!} is an isomorphism, it follows that the map $\HH(X) \to \RR\Hom_k (\HH(X),\calO_k)$ induced by $\alpha \circ \hat \mu_X \circ \zeta_X$ is an isomorphism.
	\end{proof}

In order to show that $(\HH(X), B_X)$ gives the $\bA^1$-Euler characteristic, we will use that the algebra structure on $\HH(X)$ is compatible with the algebra structure on differential forms: let $\T_X$ denote the Koszul complex as in Definition \ref{koszuleuler}.

	\begin{tm}[Hochschild--Kostant--Rosenberg Theorem]\label{HnmuisKoszul}
Let $k$ be a field of characteristic $0$. Giving $\calH^*(\HH^X(X))$ the algebra structure induced from $\LL \Delta^* \mu_X$ and $\T_X$ the algebra structure from wedge product of forms, there is an algebra isomorphism $\calH^*(\HH^X(X)) \to \T_X$.
 	\end{tm}

See also \cite[Corollaire 1.2]{Toen_Vezzosi11}, identifying the relevant $\Tor$
in commutative differential graded algebras.

	\begin{proof}
Let $I$ denote the ideal sheaf of the closed immersion $\Delta$. There is an exact sequence 
\[
0 \to I \to \calO_{X \times X} \to \Delta_* \calO_X \to 0 
\]
of coherent sheaves on $X\times X$. Applying $\LL \Delta^*$, we obtain the distinguished triangle 
\[
\LL \Delta^* I \to \LL \Delta^* \calO_{X \times X} \to \HH^X (X) \to  \LL \Delta^* I[1].
\]
The associated long exact sequence on homology sheaves defines a map \begin{equation}\label{calH1HHtoK1}
\calH\calH_1^X(X) \cong \calH^1(\HH^X(X)) \to \calH^0 \LL \Delta^* I \cong \Delta^* I \cong I/I^2 \cong \Omega_{X/k},
\end{equation}
where $\calH\calH_1$ denotes the first Hochschild homology and $\calH^1(\HH^X(X))$ denotes the first homology sheaf of the complex $\HH^X(X)$ (which is denoted with a superscript $1$ because our complexes are graded cohomologically).  $\calH^1(\LL \Delta^* \calO_{X \times X}) = 0$ because $\calO_{X \times X}$ is a flat $\calO_{X \times X}$-module.  Thus \eqref{calH1HHtoK1} is injective. Moreover the map $\calO_X \cong \Delta^* \calO_{X \times X} \to \Delta^* \Delta_*\calO_X \cong
 \calO_X$ is the identity. We therefore have that \eqref{calH1HHtoK1} is an isomorphism.

The map $\LL \Delta^* \mu_X$ induces a multiplication on $\calH (\HH^X(X))$ compatible with the multiplication an open affines of \cite{HKR62}. By \cite[Theorem 3.1]{HKR62}, $\calH (\HH^X(X))$ is the exterior algebra on $\calH^1 (\HH^X(X))$ after restricting to affine open subsets of $X$. Thus $\calH (\HH^X(X))$ is the exterior algebra on $\calH^1 (\HH^X(X))$. Therefore, \eqref{calH1HHtoK1} extends to an algebra homomorphism $\calH^*(\HH^X(X)) \to \T_X$. Since $\T_X$ is the exterior algebra on $\Omega_{X/k}$, the homomorphism $\calH^*(\HH^X(X)) \to \T_X$ is an isomorphism.
	\end{proof}

We now show that the $\bA^1$-Euler characteristic is represented by the form on $\HH(X)$ just constructed. Given a bilinear form on a graded vector space
$V^{\bullet} \times V^{\bullet} \to k$, the $0$-th graded piece $V^0$, and the
direct sums $V^{(n)} := V^{-n} + V^n$ for $n > 0$ carry a nondegenerate bilinear form. We
denote the corresponding elements of $\GW(k)$ by $[V^0]$ and $[V^{(n)}]$
respectively. The class in $\GW(k)$ determined by the form is the alternating sum
$[V^0] + \sum_{n>0} (-1)^n [V^{(n)}]$.

\begin{definition} \label{GWassociated_to_bilinear_form_on_complex}
Given a nondegenerate, symmetric $\beta: V^{\bullet} \times V^{\bullet} \to \calO_k$ in $\Drcat(k)$, let $[\beta]$ denote the class in $\GW(k)$ given by $[H^0(V^{\bullet})] + \sum_{i>0} (-1)^i [H^{\ast}(V^{\bullet})^{(i)}]$.
\end{definition}

In the following theorem, we prove that $[B_X] = \chi^{\bA^1}(X)$ by identifying the former with the pairing on $(\Hdg(X/k),\Tr)$. It is proved analogously to the proof of \cite[Proposition 2.4]{bachmann_wickelgren}. However, we also need a key duality result, which gives us a better understand of the constituents used to define $B_X$.

	\begin{tm}[{\cite[Proposition 2.4.2]{atjll}}{\cite[Theorem 3.5]{Neeman_GDHH}},{\cite[Proposition 4.6.3]{lipman1987residues}}] \label{Hnalpha'iso}
Let $d = \dim X$, and let $\alpha': \HH^X(X) \to p^!\cO_k$ be as defined above. Then the map $\mathcal H^d(\alpha')$ induced on $d$-th cohomology sheaves over $X$ is an isomorphism.
	\end{tm}

With this result, we have the following:

	\begin{theorem}\label{thm:B_XisChiA1_Xsmpr}
Let $X$ be smooth and projective over a field $k$ of characteristic $0$. Then the class of $B_X$ in $\GW(k)$ is the categorical $\bA^1$-Euler characteristic $\chi^{\bA^1}(X)$ of $X$. When $X$ is smooth and proper, the class of $B_X$ is the Euler number of the tangent bundle.
	\end{theorem}

	\begin{remark}\label{rmk:Euler_class_tangent_bundle_categorical_Euler_characteristic}
The Euler class of a vector bundle was defined by \cite{BargeMorel} and developed further by J. Fasel \cite{FaselGroupesCW} and others. Under orientation conditions on the vector bundle which are always satisfied by the tangent bundle, the Euler class can be pushed forward to an Euler number. It is a theorem of M. Levine that for $X$ smooth and projective, the categorical $\bA^1$-Euler characteristic equals the Euler number of the tangent bundle \cite[Theorem 3.1 or ArXiv version 3 Theorem 1.1]{Levine-EC}, and in particular, $X$ is dualizable in an appropriate sense, see loc. cit.
	\end{remark}

	\begin{proof}
Let $n=\dim X$. There is a hypercohomology spectral sequence $$E_2^{i,j} = \RR^ip_* H^j \HH^X(X) \Rto H^{i+j} \HH(X).$$ By the isomorphism $H^j \HH^X(X) \cong \Omega^{-j}_X$ of the Hochschild--Kostant--Rosenberg theorem,  $E_2^{i,j} \cong H^i(X,\Omega^{-j}_X)$. The hypercohomology spectral sequence is multiplicative, giving a bilinear form on the $E_r$-pages for $r\geq 2$ coming from the cup product
\[
E_2^{i,j} \times E_2^{i',j'} \to \RR^{i+i'}p_* H^{j+j'}(\HH^X(X) \otimes \HH^X(X)).
\] 
Composing with $\alpha \circ (\RR p_*(\hat \mu)) \circ \zeta_X$, we obtain a new form on the $E_r$-page which we denote $\beta_r$
\[
\beta_r: E_r^{i,j} \times E_r^{n-i,n-j} \to k.
\]
Note that $E_2^{i,j}$ is the degree $(i,j)$ summand of $(\Hdg(X/k),\Tr)$, as defined in (\ref{eq:Hdg}). By Theorems \ref{Hnalpha'iso} and \ref{HnmuisKoszul}, $\beta_2$ is the cup product pairing, whence $[\beta_2] = [(\Hdg(X/k),\Tr)]$. Recall that the notation $[\beta_2]$ was defined in \ref{GWassociated_to_bilinear_form_on_complex}.  By \cite[Lemma 2.7]{bachmann_wickelgren}, it follows that $\beta_r$ is symmetric, non-degenerate and $[\beta_r]=[\beta_2]$ for all $r \geq 2$, including $r=\infty$. By \cite[Lemma 2.6]{bachmann_wickelgren}, $[\beta_{\infty}] = B_X$. For $X$ smooth and projective $ [(\Hdg(X/k),\Tr)] = \chi^{\AA^1}(X)$ by \cite[Theorem 1.3]{levine_raksit}.  By \cite[Theorem 1.1,second Corollary p. 3]{bachmann_wickelgren} $ [(\Hdg(X/k),\Tr)]$ is the Euler number of the tangent bundle. Combining, we see $[\beta_2] = [\beta_r] = B_X$, which completes the proof by the previous. 
	\end{proof}

Using Bittner's presentation, we now show that $[X] \rgoesto [B_X]$ defines a homomorphism out of the Grothendieck ring of varieties, and which agrees with the usual Euler characteristic.  In future work we hope to prove this theorem using scdh-descent and Hochschild homology. It is interesting to note that, by contrast, Hochschild homology does not satisfy cdh-descent. The below proof is similar to one of O. R\"ondigs in that it uses cut and paste properties of the categorical Euler characteristic \cite[Theorem 5.2]{rondigs} \cite[3 Remark 2.30]{morelvoevodsky1998}. The below is of a more computational flavor, relying on recent work of M. Levine. 

	\begin{theorem}\label{chi_c_cut_and_paste_char_0}
Let $k$ be a field of characteristic $0$. The map $X \rgoesto [B_X]$ for $X$ smooth and projective defines a homomorphism $K_0(\Var_k) \rto \GW(k)$ which agrees with the categorical $\AA^1$-Euler characteristic.
	\end{theorem}

	\begin{proof}
That $[B_X]$ is the categorical $\AA^1$-Euler characterstic $\chi^{\AA^1}(X)$ is Theorem \ref{thm:B_XisChiA1_Xsmpr}. (Also see Remark \ref{rmk:Euler_class_tangent_bundle_categorical_Euler_characteristic} for context.)

Let 
\begin{equation*}
\begin{tikzcd}[column sep=large, row sep=large]
E \arrow[r]  &  \Bl_{Y} X \arrow[d] \\
Y  \arrow[r] \arrow[u,leftarrow]& X,
\end{tikzcd}
\end{equation*} 
be the blow-up diagram of a smooth, proper $k$-variety $X$ along
a smooth closed subvariety $Y$ of codimension $c$. It follows from
\cite[ Proposition 1.4 (5) or ArXiv version 3 Proposition 1.10 (4)]{Levine-EC} that $\chi^{\AA^1}( \Bl_{Y} X) = \chi^{\AA^1}(X) + \sum_{i=1}^{c-1} \langle -1 \rangle^i \chia (Y)$. The exceptional divisor $E$ is the projectivization $\mathbb{P} N_Y X$ of the normal bundle of $Y$ in $X$. By \cite[Proposition 1.4 (4) or ArXiv version 3 Proposition 1.10 (3)]{Levine-EC}, $\chi^{\AA^1} E = \sum_{i=0}^{c-1} \langle -1 \rangle^i \chia (Y)$. Thus 
\[
\chi^{\AA^1}( \Bl_{Y} X) =  \chi^{\AA^1}(X)-\chi^{\AA^1}(Y)+\chi^{\AA^1}(E)
\]
and the homomorphism is well-defined by Bittner's presentation (see Theorem~\ref{bittnerpres}).
	\end{proof}

\section{A direct perspective}\label{Example} 
In this section, we compute $\chi^{\AA^1}(X) = \chi_c^{\AA^1}(X)$ for the smooth, proper varieties $X = \Spec k$, $\PP^1$, $\PP^2$, and $ \Bl_{0}\PP^{2}$
using the complex with duality $(\Hdg(X/k), \Tr)$ described in Definition \ref{defHdg}. In other words, we compute $\AA^1$-Euler characteristics
explicitly using coherent duality. This allows a direct verification that the
association of the complex with duality $(\Hdg(X/k), \Tr)$ to a smooth
projective variety $X$ satisfies the relation $[(\Hdg(\PP^2/k), \Tr)] +
[(\Hdg(\PP^1/k), \Tr)] = [(\Hdg(\Bl_0\PP^2/k), \Tr)] + [(\Hdg(k/k), \Tr)] $ in
$\GW(k)$ imposed on classes of smooth projective varieties in Bittner's
presentation (see \ref{bittnerpres}). It gives an alternate proof of the $n=1,2$
case of M. Hoyois's computation of $\chi^{\AA^1}(\PP^n)$ \cite[Example
1.7]{hoyois2015quadratic}, and an alternate proof of the case $X=
\Bl_{0}\PP^{2}$ of M. Levine's computation of the $\AA^1$-Euler characteristic
of a blow up \cite[Proposition 1.4 or ArXiv version 3 Proposition 1.10]{Levine-EC}.  As before, we work over a
field of characteristic $0$.

The collapse of appropriate spectral sequences renders $(\Hdg(X/k), \Tr)$ equivalent to the complex with duality $(\RR{p_X}_* \T_X, \RR{p_X}_* \beta_X)$ described in \eqref{equ:p_*Beta_XKoszul}. Let $p_X: X \to \Spec k$ denote the structure map of $X$; recall that
\[
\T_X\defeq \oplus^0_{i=-d}(\wedge^{-i} T^*_X)
\]
is equipped with the natural duality given by composing multiplication of forms with projection off of the top wedge power of $T^*_X$. After pushforward to $\Spec k$, this complex with duality represents $\chi^{\bA^1}$ for smooth projective varieties, by \cite[Proposition 2.4]{bachmann_wickelgren}.

Consider the blow-up square
\begin{equation}\label{bus}
\begin{tikzcd}[column sep=large, row sep=large]
\PP^{1} \arrow[r,"i'"]  &  \Bl_{0}\PP^{2} \arrow[d,"f"] \\
\Spec k  \arrow[r,"i"] \arrow[u,leftarrow,"f'"]& \PP^{2}.
\end{tikzcd}
\end{equation}
The goal is to show the equality
\begin{equation}\label{pbus}
[\RR{p_{\PP^1}}_* \T_{\PP^1}] + [\RR{p_{\PP^2}}_* \T_{\PP^2}] = [\RR{p_{\Spec k}}_*\T_{\Spec k}] + [\RR{p_{\Bl_0 \PP^2}}_* \T_{\Bl_0\PP^2}]
\end{equation} 
$\GW(k)$; we have omitted the dualities from the notation for legibility, but these too must be considered.

	\begin{remark} 
In this section we write $p$ for the map $p_X:X \rto \Spec k$ whenever $X$ is clear from context. 
	\end{remark}

\subsection{Verification in $K_0(k)$}
In this subsection, $k$ is any field. As a $1$-category, $\operatorname{hCoh}(k)$ is equivalent to the category of graded $k$-vector spaces by taking homology. We now compute $\RR p_*\T_{\Spec k}$, $\RR p_*\T_{\PP^1}$, $\RR p_*\T_{\PP^2} $, and $\RR p_* \T_{\Bl_0 \PP^2}$ as $k$-vector spaces using the hypercohomology spectral sequence \cite[Theorem 12.12]{mccleary01}
\begin{equation}\label{SS}
E^{i,j}_{2}=\Hy^{i}\bigl(X,\T_{X}^{j}\bigr)=\Hy^{i}\bigl(X,\wedge^{j}T^{*}_{X}\bigr)  \quad \Rto \quad \Hy^{j-i}\bigl((p_{X})_{*}\T_{X}\bigr).
\end{equation}
Along the way, we verify \eqref{pbus} in $K_0(k)$. 

Consider (\ref{pbus}). The complex $\RR p_{*}\T_{\Spec k}$ is $k$ in degree zero and trivial otherwise.

The computations of $\RR p_*\T_{\PP^1}$, $\RR p_*\T_{\PP^2} $, and $\RR p_* \T_{\Bl_0 \PP^2}$ as graded $k$-vector spaces are done in Propositions~\ref{P1k2}, \ref{P2k3}, and \ref{Bk4}, respectively; we compute the forms in Subsection \ref{forms}.  We will not explicitly discuss differentials in the spectral sequence, as each of the examples we compute will end up concentrated in a single diagonal, and therefore no differentials are possible.

	\begin{proposition}\label{P1k2}
$\RR p_*\T_{\PP^1} \cong k^2$ where $k^{2}$ is concentrated in degree $0$.
	\end{proposition}
	\begin{proof}
Consider the $E_2$-page of \eqref{SS} when $X = \PP^1$, illustrated in Figure \ref{ssP1}. Noting that $T^{*}_{\PP^{1}} \cong \sh{O}(-2)$, Serre duality implies that
\[
\Hy^{1}\bigl(\PP^{1},\sh{O}_{\PP^{1}}\bigr)=\Hy^{0}\bigl(\PP^{1},\sh{O}(-2)\bigr)=0
\]
and
\[
\Hy^{1}\bigl(\PP^{1},\sh{O}(-2)\bigr)= \Hy^{0}\bigl(\PP^{1},\sh{O}_{\PP^{1}}\bigr)=k.
\]
Thus the spectral sequence is as illustrated in Figure \ref{ssP1r}.
	\end{proof}

\begin{figure}[h!]
\begin{minipage}[b]{0.45\linewidth}
\centering
\begin{tikzpicture}[scale=1]
\draw[-stealth] (0.5,0.5) -- (8,0.5);
\draw[-stealth] (0.5,0.5) -- (0.5,4);
\draw[black!90]  
         (2.1,0.3) node[below] {\tiny{$0$}}
         (5,0.3) node[below] {\tiny{$1$}}
         (7,0.3) node[below] {\tiny{$2$}}
         (0.3,1) node[left] {\tiny{$0$}}
         (0.3,2) node[left] {\tiny{$1$}}
         (0.3,3) node[left] {\tiny{$2$}};
\draw (8,0.5) node[below] {$i$}
         (0.5,4) node[left] {$j$};
\draw[very thick] 
(2.1,1) node {$\Hy^{0}\bigl(\PP^{1},\sh{O}_{\PP^{1}}\bigr)$}
(2.1,2) node {$\Hy^{0}\bigl(\PP^{1},\sh{O}(-2)\bigr)$}
(5,1) node {$\Hy^{1}\bigl(\PP^{1},\sh{O}_{\PP^{1}}\bigr)$}
(5,2) node {$\Hy^{1}\bigl(\PP^{1},\sh{O}(-2)\bigr)$};
\draw[black!90,very thick] 
                                (2.1,3) node {\large$0$}
                                (5,3) node {\large$0$}
                                (7,1) node {\large$0$}
                                (7,2) node {\large$0$}
                                (7,3) node {\large$0$};                                                                
\end{tikzpicture}
\caption{\eqref{SS} for $X = \PP^1$}
\label{ssP1}
\end{minipage}
\hspace{1.4cm}
\begin{minipage}[b]{0.45\linewidth}
\centering
\begin{tikzpicture}[scale=1]
\draw[-stealth] (0.5,0.5) -- (4,0.5);
\draw[-stealth] (0.5,0.5) -- (0.5,4);
\draw[black!90]  
         (1,0.3) node[below] {\tiny{$0$}}
         (2,0.3) node[below] {\tiny{$1$}}
         (3,0.3) node[below] {\tiny{$2$}}
         (0.3,1) node[left] {\tiny{$0$}}
         (0.3,2) node[left] {\tiny{$1$}}
         (0.3,3) node[left] {\tiny{$2$}};
\draw (4,0.5) node[below] {$i$}
         (0.5,4) node[left] {$j$}; 

\draw[very thick] (1,1) node {\large$k$}
                                          (2,2) node {\large$k$};
\draw[black!90,very thick] 
                                (1,2) node {\large$0$}
                                (1,3) node {\large$0$}
                                (2,1) node {\large$0$}
                                (2,3) node {\large$0$}
                                (3,1) node {\large$0$}
                                (3,2) node {\large$0$}
                                (3,3) node {\large$0$};                                                                
\end{tikzpicture}
\caption{Result of computation}
\label{ssP1r}
\end{minipage}
\end{figure}

	\begin{proposition}\label{P2k3} 
$\RR p_*\T_{\PP^{2}} \cong k^{3}$ where $k^{3}$ is concentrated in degree zero.
	\end{proposition} 
	\begin{proof}
The reader can refer to Figure \ref{ssP2} for the $E_{2}$-page of \eqref{SS} when $X=\PP^{2}$.

The computation of $\RR p_*\T_{\PP^{2}}$ is similar to the computation of $\RR p_*\T_{\PP^{1}}$. In this case the Koszul complex takes the form
\[
\T_{\PP^{2}}=\sh{O}_{\PP^{2}}\oplus T^{*}_{\PP^{2}}\oplus \wedge^{2}
  T^{*}_{\PP^{2}}.
\]
Thus we have possibly nonzero terms $E_{2}^{i,j}$ for $i,j \in \{0,1,2\}$.

Recall that $\Hy^{i}\bigl(\PP^{k},\sh{O}(n)\bigr)=0$ for all $n$ and all $0<i<k$. Note also that $\wedge^{2} T^{*}_{\PP^{2}} \cong \sh{O}(-3)$. It follows that $\Hy^{1}\bigl(\PP^{2},\wedge^2 T^*_{\PP^2}\bigr)$ and $\Hy^{1}\bigl(\PP^{2},\sh{O}_{\PP^{2}}\bigr)$ are trivial.  Moreover, Serre duality implies that
\[
\Hy^{2}\bigl(\PP^{2},\sh{O}_{\PP^{2}}\bigr)=\Hy^{0}\bigl(\PP^{2},\sh{O}(-3)\bigr)=0
\]
and
\[
\Hy^{2}\bigl(\PP^{2},\sh{O}(-3)\bigr)=\Hy^{0}\bigl(\PP^{2},\sh{O}_{\PP^{2}}\bigr)=k.
\]

Since $T\PP^{n}=\Hom\bigl(\sh{O}(-1),\sh{O}^{n+1}/\sh{O}(-1)\bigr)$, there is a  short exact sequence
\[
0\rightarrow \sh{O}_{\PP^{n}}\rightarrow \sh{O}(-1)^{n+1} \rightarrow T\PP^{n} \rightarrow 0.
\]
The dual sequence is
\begin{equation}\label{dualized}
0\rightarrow T^{*}_{\PP^{n}}\rightarrow \sh{O}(1)^{n+1} \rightarrow \sh{O}_{\PP^{n}} \rightarrow 0.
\end{equation}
From the long exact sequence associated to \eqref{dualized}, it
follows that 
\[
\Hy^{0}\bigl(\PP^{2},T^{*}_{\PP^{2}}\bigr)=\Hy^{2}\bigl(\PP^{2},T^{*}_{\PP^{2}}\bigr)=0$$  and $$\Hy^{1}\bigl(\PP^{2},T^{*}_{\PP^{2}}\bigr)=k.
\]
These computations are summarized in Figure \ref{ssP2r}.
	\end{proof}

\begin{figure}[h!]
\begin{minipage}[b]{0.45\linewidth}
\centering
\begin{tikzpicture}[scale=0.9]
\draw[-stealth] (0.5,0.5) -- (10.8,0.5);
\draw[-stealth] (0.5,0.5) -- (0.5,5);
\draw[black!90]  
         (2.1,0.3) node[below] {\tiny{$0$}}
         (5,0.3) node[below] {\tiny{$1$}}
         (7.9,0.3) node[below] {\tiny{$2$}}
         (9.9,0.3) node[below] {\tiny{$3$}}
         (0.3,1) node[left] {\tiny{$0$}}
         (0.3,2) node[left] {\tiny{$1$}}
         (0.3,3) node[left] {\tiny{$2$}}
         (0.3,4) node[left] {\tiny{$3$}};
\draw (10.8,0.5) node[below] {$i$}
         (0.5,5) node[left] {$j$};
\draw[very thick] 
(2.1,1) node {\small$\Hy^{0}\bigl(\PP^{2},\sh{O}_{\PP^{2}}\bigr)$}
(2.1,2) node {\small$\Hy^{0}\bigl(\PP^{2},T^{*}_{\PP^{2}}\bigr)$}
(2.1,3) node {\small$\Hy^{0}\bigl(\PP^{2},\sh{O}(-3)\bigr)$}
(5,1) node {\small$\Hy^{1}\bigl(\PP^{2},\sh{O}_{\PP^{2}}\bigr)$}
(5,2) node {\small$\Hy^{1}\bigl(\PP^{2},T^{*}_{\PP^{2}}\bigr)$}
(5,3) node {\small$\Hy^{1}\bigl(\PP^{2},\sh{O}(-3)\bigr)$}
(7.9,1) node {\small$\Hy^{2}\bigl(\PP^{2},\sh{O}_{\PP^{2}}\bigr)$}
(7.9,2) node {\small$\Hy^{2}\bigl(\PP^{2},T^{*}_{\PP^{2}}\bigr)$}
(7.9,3) node {\small$\Hy^{2}\bigl(\PP^{2},\sh{O}(-3)\bigr)$};
\draw[black!90,very thick] 
                                (2.1,4) node {\large$0$}
                                (5,4) node {\large$0$}
                                (7.9,4) node {\large$0$}
                                (9.9,1) node {\large$0$}
                                (9.9,2) node {\large$0$}
                                (9.9,3) node {\large$0$}
                                (9.9,4) node {\large$0$};                                                                
\end{tikzpicture}
\caption{\eqref{SS} for $X = \PP^2$.}
\label{ssP2}
\end{minipage}
\hspace{1.4cm}
\begin{minipage}[b]{0.45\linewidth}
\centering
\begin{tikzpicture}[scale=0.8]
\draw[-stealth] (0.5,0.5) -- (5,0.5);
\draw[-stealth] (0.5,0.5) -- (0.5,5);
\draw[black!90]  
         (1,0.3) node[below] {\tiny{$0$}}
         (2,0.3) node[below] {\tiny{$1$}}
         (3,0.3) node[below] {\tiny{$2$}}
         (4,0.3) node[below] {\tiny{$3$}}
         (0.3,1) node[left] {\tiny{$0$}}
         (0.3,2) node[left] {\tiny{$1$}}
         (0.3,3) node[left] {\tiny{$2$}}
         (0.3,4) node[left] {\tiny{$3$}};
\draw (5,0.5) node[below] {$i$}
         (0.5,5) node[left] {$j$};
\draw[very thick] (1,1) node {\large$k$}
                         (2,2) node {\large$k$}
                         (3,3) node {\large$k$};
\draw[black!90,very thick] 
                                (1,2) node {\large$0$}
                                (1,3) node {\large$0$}
                                (1,4) node {\large$0$}
                                (2,1) node {\large$0$}
                                (2,3) node {\large$0$}
                                (2,4) node {\large$0$}
                                (3,1) node {\large$0$}
                                (3,2) node {\large$0$}
                                (3,4) node {\large$0$}
                                (4,1) node {\large$0$}
                                (4,2) node {\large$0$}
                                (4,3) node {\large$0$}
                                (4,4) node {\large$0$};                                                                
\end{tikzpicture}
\caption{Result of computation.}
\label{ssP2r}
\end{minipage}
\end{figure}

	\begin{proposition}\label{Bk4}
$\RR p_*\T_{\Bl_{0}\PP^{2}} \cong k^{4}$ where $k^{4}$ is concentrated in degree zero.
\end{proposition}

In order to prove this we must first carry out two auxilliary computations in \ref{Lem1} and \ref{Lem2}.

	\begin{lemma} \label{Lem1}
\[\
Hy^{i}\bigl(\Bl_{0}\PP^{2};\Omega_{\Bl_{0}\PP^{2}/\PP^{2}}\bigr) \cong
    \begin{cases}
     k & i = 1,\\
     0 & i \neq 1.
   \end{cases}
\]
	\end{lemma}
	\begin{proof}
Let $i':\PP^{1} \rightarrow \Bl_{0}\PP^{2}$ be the inclusion of the exceptional divisor.  Since K\"ahler differentials commute with pullback \cite[Tag 01UM]{stacksproject}, the support of $\Omega_{\Bl_{0}\PP^{2}/\PP^{2}}$ is contained in the exceptional divisor. In fact, the canonical map $\Omega_{\Bl_{0}\PP^{2}/\PP^{2}}\to i'_{*}(i')^* \Omega_{\Bl_{0}\PP^{2}/\PP^{2}}$ is an isomorphism. To see this: let $\calI$ denote the sheaf of ideals associated to the closed immersion $i'$.  By the proof of  \cite[Tag 01QY]{stacksproject}, it suffices to see that $\calI  \Omega_{\Bl_{0}\PP^{2}/\PP^{2}} = 0$. Since the support of $\Omega_{\Bl_{0}\PP^{2}/\PP^{2}}$ is contained in the exceptional divisor, it thus suffices to see that $ \calI \Omega_{\Bl_{0}\bA^{2}/\bA^{2}} = 0$. This can be seen by direct computation. The blow-up
\[
\Bl_0 \bA^2 \cong \Proj_{k[x,y]} \frac{k[x,y][Z,W]}{\langle xZ-yW\rangle}
\]
is covered by two affine opens, one canonically isomorphic to $\Spec \frac{k[x,y,\frac{W}{Z}]}{\langle x - y \frac{W}{Z} \rangle}$ and the other canonically isomorphic to $\Spec \frac{k[x,y,\frac{Z}{W}]}{\langle x \frac{Z}{W}- y  \rangle}$. The sheaf of relative differentials of, say, the first over $\Spec k[x,y]$ has a single generator $d\frac{W}{Z}$ with a single relation $-yd\frac{W}{Z}$. The sheaf of differentials of the second is computed similarly, showing that $ \calI \Omega_{\Bl_{0}\bA^{2}/\bA^{2}} = 0$.  Thus  $\Omega_{\Bl_{0}\PP^{2}/\PP^{2}}\cong i'_{*}(i')^* \Omega_{\Bl_{0}\PP^{2}/\PP^{2}}$ as claimed. 
  
It follows that $\Omega_{\Bl_{0}\PP^{2}/\PP^{2}}\cong i'_{*}T^{*}_{\PP^{1}}$ by another application of the commutativity of K\"ahler differentials and pullback\cite[Tag 01UM]{stacksproject}. Thus
\[
\Hy^{i}\bigl(\Bl_{0}\PP^{2};\Omega_{\Bl_{0}\PP^{2}/\PP^{2}}\bigr) \cong \Hy^{i}\bigl(\Bl_{0}\PP^{2}; i'_{*}T^{*}_{\PP^{1}}\bigr) \cong \Hy^{i}\bigl(\PP^{1};T^{*}_{\PP^{1}}\bigr).
\] 
The result then follows from Proposition \ref{P1k2}.
	\end{proof}

	\begin{lemma} \label{Lem2}
\[
\Hy^{i}\bigl(\Bl_{0}\PP^{2};f^{*}T^{*}_{\PP^{2}}\bigr) \cong
    \begin{cases}
     k & i =1, \\
     0 & i \neq 1.
    \end{cases}
\]
	\end{lemma}
	\begin{proof}
Consider the blow-up square \eqref{bus}. 
	\begin{claim}\label{2.9claim} 
This gives a distinguished triangle in the bounded derived category of $\PP^2$:
\[
\cO_{\PP^2} \to i_* \cO_{\Spec k} \oplus \RR f_* \cO_{\Bl_0\PP^2} \to \RR m_*\cO_{\PP^1},
\]
where $m:= f\circ i' = i \circ f'.$  
	\end{claim}

Given the claim, note that $T_{\PP^2}^*$ is a flat module (in fact, locally free) since $\PP^2$ is smooth, so tensoring preserves exact sequences. Thus we get a distinguished triangle: 
\[
 T^*_{\PP^2} \to  (T^*_{\PP^2}\otimes \RR i_* \cO_{\Spec k}) \oplus  (T^*_{\PP^2} \otimes \RR f_* \cO_{\Bl_0\PP^2}) \to  T^*_{\PP^2}\otimes \RR m_*\cO_{\PP^1},
\]
Now, using a projection formula from \cite[Theorem 20.49.2]{stacksproject} we have that 
\[
T^*_{\PP^2} \otimes \RR f_* \cO_{\Bl_0\PP^2} \simeq \RR f_*(f^*T_{\PP^2}^*),
\]
and similarly for other terms. Thus we have a distinguished triangle:

\[
T^*_{\PP^2} \to \RR i_*(i^*T_{\PP^2}^*) \oplus  \RR f_*(f^*T^*_{\PP^2}) \to  \RR m_*(m^*T^*_{\PP^2}).
\]
Now, we apply $\RR ({p_{\PP^2}}_*)$ to get a long exact sequence on cohomology.

\begin{equation}\label{les2}
\begin{tikzcd}[column sep=small, row sep=large]
\Hy^{0}\bigl(\PP^{2};T^{*}_{\PP^{2}}\bigr) \arrow{r} &  \Hy^{0}\bigl(\Spec k;i^{*}T^{*}_{\PP^{2}}\bigr)\oplus \Hy^{0}\bigl(\Bl_{0}\PP^{2};f^{*}T^{*}_{\PP^{2}}\bigr) \arrow{r}& \Hy^{0}\bigl(\PP^{1};(i\circ f')^{*}T^{*}_{\PP^{2}}\bigr) \arrow[lld,rounded corners,to path={ -- (5.3,1.02) \tikztonodes -- ([xshift=-62.8ex,yshift=-2.9ex]\tikztostart.south) -- (\tikztotarget.north)}]\\
\Hy^{1}\bigl(\PP^{2};T^{*}_{\PP^{2}}\bigr) \arrow{r} &  \Hy^{1}\bigl(\Spec k;i^{*}T^{*}_{\PP^{2}}\bigr)\oplus \Hy^{1}\bigl(\Bl_{0}\PP^{2};f^{*}T^{*}_{\PP^{2}}\bigr) \arrow{r}& \Hy^{1}\bigl(\PP^{1};(i\circ f')^{*}T^{*}_{\PP^{2}}\bigr)\arrow[lld,rounded corners,to path={ -- (5.3,-0.94) \tikztonodes -- ([xshift=-62.8ex,yshift=-2.9ex]\tikztostart.south) -- (\tikztotarget.north)}]\\
\Hy^{2}\bigl(\PP^{2};T^{*}_{\PP^{2}}\bigr) \arrow{r} &  \Hy^{2}\bigl(\Spec k;i^{*}T^{*}_{\PP^{2}}\bigr)\oplus \Hy^{2}\bigl(\Bl_{0}\PP^{2};f^{*}T^{*}_{\PP^{2}}\bigr) \arrow{r}& \Hy^{2}\bigl(\PP^{1};(i\circ f')^{*}T^{*}_{\PP^{2}}\bigr).
\end{tikzcd}
\end{equation}
Considering dimensions, we see that
\[
\Hy^{1}\bigl(\Spec k;i^{*}T^{*}_{\PP^{2}}\bigr)=\Hy^{2}\bigl(\Spec
k;i^{*}T^{*}_{\PP^{2}}\bigr)=\Hy^{2}\bigl(\PP^{1};(i\circ
f')^{*}T^{*}_{\PP^{2}}\bigr)=0,
\] 
implying that $H^2(\Bl_0\PP^2;f^* T^*_{\PP^2}) = 0.$

From the proof Proposition \ref{P2k3} we know
\[
\Hy^{j}\bigl(\PP^{2};T^{*}_{\PP^{2}}\bigr) = 0
\]
unless $j = 1$, in which case it is $k$.

Recall that $i:\Spec k \hookrightarrow \PP^{2}$ is the inclusion. By definition
$i^{*}T^{*}_{\PP^{2}}=i^{-1}T^{*}_{\PP^{2}}\otimes_{i^{-1}\sh{O}_{\PP^{2}}}\sh{O}_{k}$. Therefore
\[
i^{*}T^{*}_{\PP^{2}}=T^{*}_{\PP^{2}}|_{\Spec
  k}\otimes_{\sh{O}_{\PP^{2}}|_{\Spec k}}\sh{O}_{k}\cong T^{*}_{\PP^{2}}|_{\Spec
  k} \cong \sh{O}_{k}^{\oplus 2}.
\]
This implies
\[
\Hy^{0}\bigl(\Spec k;i^{*}T^{*}_{\PP^{2}}\bigr)=\Hy^{0}\bigl(\Spec k;\sh{O}_{k}^{\oplus 2}\bigr)=\Hy^{0}\bigl(\Spec k;\sh{O}_{k}\bigr)\oplus\Hy^{0}\bigl(\Spec k;\sh{O}_{k}\bigr)=k^{2}.
\]
Observing that $(i\circ f')^{*}T^{*}_{\PP^{2}}\cong (f')^{*}\sh{O}_{k}^{\oplus 2} \cong \sh{O}_{\PP^{1}}^{\oplus 2}$, we get that: 
\[
\Hy^{0}\bigl(\PP^{1};(i\circ f')^{*}T^{*}_{\PP^{2}}\bigr)=\Hy^{0}\bigl(\PP^{1};\sh{O}_{\PP^{1}}^{\oplus 2}\bigr)=k^{2},
\]
and
\[
\Hy^{1}\bigl(\PP^{1};(i\circ f')^{*}T^{*}_{\PP^{2}}\bigr)=\Hy^{1}\bigl(\PP^{1};\sh{O}_{\PP^{1}}^{\oplus 2}\bigr)=\Hy^{1}\bigl(\PP^{1};\sh{O}_{\PP^{1}}\bigr)\oplus \Hy^{1}\bigl(\PP^{1};\sh{O}_{\PP^{1}}\bigr)=0.
\]

Plugging this in to \eqref{les2} gives the exact sequence
\[
0 \rto  k^2 \oplus  \Hy^{0}\bigl(\Bl_{0}\PP^{2};f^{*}T^{*}_{\PP^{2}}\bigr)
  \rto k^2 \rto k \rto \Hy^{1}\bigl(\Bl_{0}\PP^{2};f^{*}T^{*}_{\PP^{2}}\bigr) \rto 0.
\]
The desired result follows.
	\end{proof}

	\begin{proof}[Proof of Claim ~\ref{2.9claim}] 
Consider the short exact sequence of $\cO_{\Bl_0\PP^2}$-modules

\begin{equation*}\label{SES1} 
0 \to \cO_{\Bl_0\PP^2}(-E) \to \cO_{\Bl_0\PP^2} \to i_*' \cO_{E}\to 0, \end{equation*} 
where $E \simeq \PP^1$ is the exceptional divisor of the blow-up. Applying $Rf_*$ gives an exact triangle in the derived category of $\PP^2$

\begin{equation}\label{distinguishedtriangle1} 
\RR f_* \cO_{\Bl_0\PP^2}(-E) \to \RR f_* \cO_{\Bl_0\PP^2} \to \RR f_* (i_*' \cO_{E}), \end{equation} 

whose associated long exact sequence on cohomology gives
\[ 
0 \to f_* \cO_{\Bl_0\PP^2}(-E) \to f_*\cO_{\Bl_0\PP^2}\simeq \cO_{\PP^2} \to m_*\cO_{E} \simeq i_*\cO_{\Spec k} \to \RR^1f_*\cO_{\Bl_0\PP^2}(-E).
\]
But in fact $\RR^if_*\cO_{\Bl_0\PP^2}(-E)=0$ for $i>0$, as the following argument shows. Note first that $\RR^i f_{*}\cO_{\Bl_0\PP^2}(-E))$ is a coherent sheaf $\forall i\geq 0$ by \cite[Theorem 30.19.1]{stacksproject}. 

It is enough to check that the fibers are all zero. Combining \cite[Theorem 3.2.1]{egaIII}, \cite[Exercise II.3.10]{HartshorneAG}, and \cite[Corollary III.9.4]{HartshorneAG}, we see that the fiber over $p\in \PP^2$ is 
\[
(\RR^if_{*}\cO_{\Bl_0\PP^2}(-E))(p)\cong H^i(\Bl_0\PP^2_p,\cO_{\Bl_0\PP^2}(-E)_{f^{-1}(p)})\cong H^i(f^{-1}(p), \cO_{\Bl_0\PP^2}(-E)).
\]

Consider the case where $p\in \PP^2-0$. Since $f|_{\PP^2-0}\colon \Bl_0\PP^2-E\to \PP^2-0$ is an isomorphism, $f^{-1}(p)$ is a point and therefore affine. Thus $H^i(f^{-1}(p), \cO_{\Bl_0\PP^2}(-E))=0$ for all $i>0$, as desired. 

Consider the case $p=0$. Then we have $f^{-1}(p)\cong \PP^1$, and $$H^1(f^{-1}(p), \cO_{\Bl_0\PP^2}(-E))\cong H^1(\PP^1, \cO_{\PP^1}(1))=0,$$ as desired. Therefore $\RR^if_{*}\cO_{\Bl_0\PP^2}(-E)=0$ for $i>0$, and so we get a short exact sequence 

\begin{equation*}\label{SES2} 
0 \to f_*\cO_{\Bl_0\PP^2}(-E) \to \cO_{\PP^2} \to i_*\cO_{\Spec k} \to 0,  \end{equation*} 
which in fact represents a distinguished triangle

\begin{equation*}\label{distinguishedtriangle3} 
\RR f_*\cO_{\Bl_0\PP^2}(-E) \to \cO_{\PP^2} \to i_*\cO_{\Spec k},
\end{equation*}
in the derived category of $\PP^2$ with the property that the shift map $i_*\cO_{\Spec k} \to \RR f_*\cO_{\Bl_0\PP^2}(-E)[1]$ is zero. By \cite[Theorem 13.4.10]{stacksproject}, we get that 
\begin{equation}\label{keyidentification}
\cO_{\PP^2} =  \RR f_*\cO_{\Bl_0\PP^2}(-E) \oplus i_*\cO_{\Spec k}
\end{equation}
in the derived category of $\PP^2.$ 

Now, note that \eqref{distinguishedtriangle1} becomes

\begin{equation}\label{distinguishedtriangle4}  
\RR f_* \cO_{\Bl_0\PP^2}(-E) \to \RR f_* \cO_{\Bl_0\PP^2} \to \RR m_* \cO_{\PP^1}.\end{equation}
The modifications to the last term follow from recalling that $E \simeq \PP^1$ and that the inclusion $i'\colon \PP^1\to X'$ is affine. By \cite[Theorem 36.5.3]{stacksproject}, $\RR i_*' = i_*'$, and hence 
\[
\RR f_*( i_*' \cO_E) =\RR f_* \RR i'_* \cO_{\PP^1}= \RR m_* \cO_{\PP^1}.
\]
Moreover, we also have a trivial distinguished triangle:

\begin{equation}\label{distinguishedtriangle5}
 i_*\cO_{\Spec k} \to i_*O_{\Spec k} \to 0. 
\end{equation}

Using \cite[Lemma 13.4.9]{stacksproject} we get a distinguished triangle by summing \eqref{distinguishedtriangle4} and \eqref{distinguishedtriangle5}: 

\begin{equation*}\label{distinguishedtriangle6}  
\RR f_* \cO_{\Bl_0\PP^2}(-E)\oplus  i_*\cO_{\Spec k} \to \RR f_* \cO_{\Bl_0\PP^2}\oplus  i_*\cO_{\Spec k} \to \RR m_* \cO_{\PP^1}.
\end{equation*}

But, appealing to \eqref{keyidentification}, this is precisely the distinguished triangle claimed.
	\end{proof}

We are now ready to prove Proposition~\ref{Bk4}.

	\begin{proof}[Proof of Proposition~\ref{Bk4}] 
Consider the top row $E^2$-page of the spectral sequence (\ref{SS}) for $\Bl_{0}\PP^{2}$; this is given in Figure~\ref{ssBl}.   Following the proof of \cite[Proposition V.3.4]{HartshorneAG} we see that $\Hy^{2}\bigl(\Bl_{0}\PP^{2},\sh{O}_{\Bl_{0}\PP^{2}}\bigr)= \Hy^{1}\bigl(\Bl_{0}\PP^{2},\sh{O}_{\Bl_{0}\PP^{2}}\bigr)=0$. Since $\Bl_{0}\PP^{2}$ is smooth and compact, Serre duality implies that
\begin{align*}
  \Hy^{2}\bigl(\Bl_{0}\PP^{2},\omega_{\Bl_{0}\PP^{2}}\bigr) &=\Hy^{0}\bigl(\Bl_{0}\PP^{2},\sh{O}_{\Bl_{0}\PP^{2}}\bigr)= k, \\
  \Hy^{1}\bigl(\Bl_{0}\PP^{2},\omega_{\Bl_{0}\PP^{2}}\bigr)&= \Hy^{1}\bigl(\Bl_{0}\PP^{2},\sh{O}_{\Bl_{0}\PP^{2}}\bigr)=0, \hbox{
  and}\\
  \Hy^{0}\bigl(\Bl_{0}\PP^{2},\omega_{\Bl_{0}\PP^{2}}\bigr)&=\Hy^{2}\bigl(\Bl_{0}\PP^{2},\sh{O}_{\Bl_{0}\PP^{2}}\bigr)=0.
\end{align*}

To compute the terms in the second row of the $E_{2}$-page, consider the
blow-up map $f:\Bl_{0}\PP^{2}\rightarrow \PP^{2}$.  There is an exact
sequence of sheaves on $\Bl_{0}\PP^{2}$,
\[
f^{*}T^{*}_{\PP^{2}}\rightarrow T^{*}_{\Bl_{0}\PP^{2}}\rightarrow \Omega_{\Bl_{0}\PP^{2}/\PP^{2}} \rightarrow 0. 
\]
In general this sequence is not left exact, but we will show that $f^{*}T^{*}_{\PP^2}\rightarrow T^{*}_{\Bl_{0}\PP^2}$ is injective in this specific example. We can explicitly write 
\[
\Bl_{0}\PP^2=\Proj\left(\frac{k[x,y][X,Y]}{\langle Xy-xY\rangle}\right)
\] 
where $x$ and $y$ have degree 0 and $X$ and $Y$ have degree 1. We will show the map is injective on the affine open sets 
\[
U_X\defeq \{X\neq 0\}
\]
and 
\[
U_Y\defeq \{Y\neq 0\}.
\]
The situation is symmetric, so we will only give the argument on $U_X$. Observe 
\[
U_X=\Spec\left(\frac{k[x,y][X,Y][1/X]_0}{\langle Xy-xY\rangle}\right)=\Spec\left(\frac{k[x,y,Y/X]}{\langle Y-xY/X\rangle}\right).
\]
Considering $f|_{U_X}\colon U_X\to \Spec k[x,y]\subseteq \PP^2$, we have an explicit description
\[
f^{*}T^{*}_{\PP^2}(U_X) = \O dx\oplus \O dy
\]
and
\[
T^{*}\Bl_{0}\PP^2(U_X) = \O dx\oplus \O d(Y/X).
\]
Under this identification, the map of sheaves $f^{*}T^{*}_{\PP^2}(U_X)\to T^{*}\Bl_{0}\PP^2(U_X)$ is given in coordinates by
\[
dx\mapsto dx
\]
and 
\[
dy\mapsto d(x\frac{Y}{X})=(dx)(\frac{Y}{X})+xd(\frac{Y}{X}).
\]
This is the map we want to show is injective. 

Observe that given $(g,h)\in \O\oplus \O$, $$(g,h)\mapsto (g+h\frac{Y}{X}, xh).$$ If $xh=0 \in \frac{k[x,y,Y/X]}{\langle Y-xY/X\rangle}$, then $h=0$. Therefore if $(0,0)=(g+h\frac{Y}{X},xh)=(g,0),$ we must have $g=0$. Thus, 
\[
\ker(f^{*}T^{*}\PP^2(U_X)\to T^{*}_{\Bl_0\PP^2}(U_X))=\{(0,0)\}, 
\]
as desired. Applying the same argument for $U_Y$, we conclude that $f^{*}T^{*}_{\PP^{2}}\rightarrow T^{*}_{\Bl_{0}\PP^{2}}$ is injective. 

It follows that there is a short exact sequence of coefficients
\[
0\rightarrow f^{*}T^{*}_{\PP^{2}}\rightarrow T^{*}_{\Bl_{0}\PP^{2}}\rightarrow \Omega_{\Bl_{0}\PP^{2}/\PP^{2}} \rightarrow 0,
\]
that induces a long exact sequence on cohomology
\begin{equation}\label{les1}
\begin{tikzcd}[column sep=small, row sep=large]
\Hy^{0}\bigl(\Bl_{0}\PP^{2};f^{*}T^{*}_{\PP^{2}}\bigr) \arrow{r} & \Hy^{0}\bigl(\Bl_{0}\PP^{2};T^{*}_{\Bl_{0}\PP^{2}}\bigr) \arrow{r}& \Hy^{0}\bigl(\Bl_{0}\PP^{2};\Omega_{\Bl_{0}\PP^{2}/\PP^{2}}\bigr) \arrow[lld,rounded corners,to path={ -- (4.1,1.01) \tikztonodes -- ([xshift=-49.7ex,yshift=-2.8ex]\tikztostart.south) -- (\tikztotarget.north)}]\\
\Hy^{1}\bigl(\Bl_{0}\PP^{2};f^{*}T^{*}_{\PP^{2}}\bigr) \arrow{r} & \Hy^{1}\bigl(\Bl_{0}\PP^{2};T^{*}_{\Bl_{0}\PP^{2}}\bigr) \arrow{r}& \Hy^{1}\bigl(\Bl_{0}\PP^{2};\Omega_{\Bl_{0}\PP^{2}/\PP^{2}}\bigr) \arrow[lld,rounded corners,to path={ -- (4.1,-0.94) \tikztonodes -- ([xshift=-49.7ex,yshift=-2.8ex]\tikztostart.south) -- (\tikztotarget.north)}]\\
\Hy^{2}\bigl(\Bl_{0}\PP^{2};f^{*}T^{*}_{\PP^{2}}\bigr) \arrow{r} &\Hy^{2}\bigl(\Bl_{0}\PP^{2};T^{*}_{\Bl_{0}\PP^{2}}\bigr) \arrow{r}& \Hy^{2}\bigl(\Bl_{0}\PP^{2};\Omega_{\Bl_{0}\PP^{2}/\PP^{2}}\bigr)\rightarrow \cdots
\end{tikzcd}
\end{equation}

Substituting the results of Lemmas~\ref{Lem1} and \ref{Lem2} implies that \eqref{les1} takes the form
\[
0 \rightarrow \Hy^{0}\bigl(\Bl_{0}\PP^{2};T^{*}_{\Bl_{0}\PP^{2}}\bigr) \rightarrow 0 \rightarrow k \rightarrow \Hy^{1}\bigl(\Bl_{0}\PP^{2};T^{*}_{\Bl_{0}\PP^{2}}\bigr) \rightarrow k \rightarrow 0 \rightarrow \Hy^{2}\bigl(\Bl_{0}\PP^{2};T^{*}_{\Bl_{0}\PP^{2}}\bigr) \rightarrow 0.
\]
Consequently, 
\[
\Hy^{0}\bigl(\Bl_{0}\PP^{2};T^{*}_{\Bl_{0}\PP^{2}}\bigr)=\Hy^{2}\bigl(\Bl_{0}\PP^{2};T^{*}_{\Bl_{0}\PP^{2}}\bigr)=0
\]
and 
\[
\Hy^{1}\bigl(\Bl_{0}\PP^{2};T^{*}_{\Bl_{0}\PP^{2}}\bigr)=k^{2}.
\]
Therefore, the second page of the spectral sequence takes the form shown in Figure \ref{ssBlr}, which means that the spectral sequence collapses on this page.
	\end{proof}

\begin{figure}[h!]
\begin{minipage}[b]{0.45\linewidth}
\centering
\begin{tikzpicture}[scale=0.9]
\draw[-stealth] (0.5,0.5) -- (11,0.5);
\draw[-stealth] (0.5,0.5) -- (0.5,5);
\draw[black!90]  
         (2.1,0.3) node[below] {\tiny{$0$}}
         (5,0.3) node[below] {\tiny{$1$}}
         (7.9,0.3) node[below] {\tiny{$2$}}
         (9.9,0.3) node[below] {\tiny{$3$}}
         (0.3,1) node[left] {\tiny{$0$}}
         (0.3,2) node[left] {\tiny{$1$}}
         (0.3,3) node[left] {\tiny{$2$}}
         (0.3,4) node[left] {\tiny{$3$}};
\draw (11,0.5) node[below] {$i$}
         (0.5,5) node[left] {$j$};
\draw[very thick] 
(2.1,1) node {\tiny$\Hy^{0}\bigl(\Bl_{0}\PP^{2},\sh{O}_{\Bl_{0}\PP^{2}}\bigr)$}
(2.1,2) node {\tiny$\Hy^{0}\bigl(\Bl_{0}\PP^{2},T^{*}_{\Bl_{0}\PP^{2}}\bigr)$}
(2.1,3) node {\tiny$\Hy^{0}\bigl(\Bl_{0}\PP^{2},\omega_{\Bl_{0}\PP^{2}}\bigr) $}
(5,1) node {\tiny$\Hy^{1}\bigl(\Bl_{0}\PP^{2},\sh{O}_{\Bl_{0}\PP^{2}}\bigr)$}
(5,2) node {\tiny$\Hy^{1}\bigl(\Bl_{0}\PP^{2},T^{*}_{\Bl_{0}\PP^{2}}\bigr)$}
(5,3) node {\tiny$\Hy^{1}\bigl(\Bl_{0}\PP^{2},\omega_{\Bl_{0}\PP^{2}}\bigr) $}
(7.9,1) node {\tiny$\Hy^{2}\bigl(\Bl_{0}\PP^{2},\sh{O}_{\Bl_{0}\PP^{2}}\bigr)$}
(7.9,2) node {\tiny$\Hy^{2}\bigl(\Bl_{0}\PP^{2},T^{*}_{\Bl_{0}\PP^{2}}\bigr)$}
(7.9,3) node {\tiny$\Hy^{2}\bigl(\Bl_{0}\PP^{2},\omega_{\Bl_{0}\PP^{2}}\bigr) $};
\draw[black!90,very thick] 
                                (2.1,4) node {\large$0$}
                                (5,4) node {\large$0$}
                                (7.9,4) node {\large$0$}
                                (9.9,1) node {\large$0$}
                                (9.9,2) node {\large$0$}
                                (9.9,3) node {\large$0$}
                                (9.9,4) node {\large$0$};                                                                
\end{tikzpicture}
\caption{\eqref{SS} for $X = Bl_0\PP^2$}
\label{ssBl}
\end{minipage}
\hspace{1.4cm}
\begin{minipage}[b]{0.45\linewidth}
\centering
\begin{tikzpicture}[scale=0.8]
\draw[-stealth] (0.5,0.5) -- (5,0.5);
\draw[-stealth] (0.5,0.5) -- (0.5,5);
\draw[black!90]  
         (1,0.3) node[below] {\tiny{$0$}}
         (2,0.3) node[below] {\tiny{$1$}}
         (3,0.3) node[below] {\tiny{$2$}}
         (4,0.3) node[below] {\tiny{$3$}}
         (0.3,1) node[left] {\tiny{$0$}}
         (0.3,2) node[left] {\tiny{$1$}}
         (0.3,3) node[left] {\tiny{$2$}}
         (0.3,4) node[left] {\tiny{$3$}};
\draw (5,0.5) node[below] {$i$}
         (0.5,5) node[left] {$j$};
\draw[very thick] (1,1) node {\large$k$}
                         (2,2) node {\large$k^{2}$}
                         (3,3) node {\large$k$};
\draw[black!90,very thick] 
                                (1,2) node {\large$0$}
                                (1,3) node {\large$0$}
                                (1,4) node {\large$0$}
                                (2,1) node {\large$0$}
                                (2,3) node {\large$0$}
                                (2,4) node {\large$0$}
                                (3,1) node {\large$0$}
                                (3,2) node {\large$0$}
                                (3,4) node {\large$0$}
                                (4,1) node {\large$0$}
                                (4,2) node {\large$0$}
                                (4,3) node {\large$0$}
                                (4,4) node {\large$0$};                                                                
\end{tikzpicture}
\caption{Result of computation}
\label{ssBlr}
\end{minipage}
\end{figure}
 
This completes the proof that equality \eqref{pbus} holds after applying the forgetful homomorphism $\GW(k) \to K_0(k)$. It remains to show the relation is satisfied by forms. 

\subsection{Verification in $\GW(k)$}\label{forms}
In this subsection, $k$ denotes a field of characteristic not $2$. We compute the bilinear form on $\RR p_* \T_X$ for $X = \Spec k, \PP^1, \PP^2$ or $\Bl_0\PP^2,$ and complete the verification of equality \eqref{pbus} in $\GW(k)$. We begin by defining elements in $\GW(k)$ which we will use to express these classes. 
\begin{definition}\label{hyperbolic} The hyperbolic bilinear form $H$ is defined to be the rank 2 symmetric bilinear form with Gram matrix 

\begin{equation*}
\left(\begin{array}{cc}
0 & 1  \\
1 & 0 
\end{array}\right).
\end{equation*}
\end{definition} 

	\begin{definition}\label{form_a} 
Let $a\in k$. Then we write $\langle a \rangle$ for the rank 1 bilinear form $(x,y) \mapsto axy.$ 
	\end{definition}
	\begin{remark}\label{H_as_dspeck} 
In $\GW(k)$ there is an equality between the class of  $H$ and $\langle 1 \rangle + \langle -1 \rangle$ in $\GW(k).$
	\end{remark}

We start with computing the form for $\RR p_{*} \T_{\Spec k}:$ recall that the complex $\RR p_{*}\T_{\Spec k}$ is $k$ in degree zero. The following result is immediate:

	\begin{proposition}\label{dspeck}
$\RR p_{*} T_{\Spec k}$ has the trivial duality $\langle 1 \rangle$. 
	\end{proposition}

Next, consider $\PP^1.$ 

	\begin{proposition}\label{dP1k2} 
 $\chi^{\bA^1}(\PP^{1})=H$ in $\GW(k)$.
	\end{proposition}

	\begin{remark}
This provides an alternate proof of the $n=1$ case of \cite[Example 1.7]{hoyois2015quadratic}.
	\end{remark}

	\begin{proof}
Referring to the spectral sequences in Figures \ref{ssP1} and \ref{ssP1r}, we recall that $$\RR p_{*} \T_{\PP^1} \simeq H^0(\PP^1, \cO_{\PP^1}) \oplus H^1(\PP^1, \Omega_{\PP^1}) \simeq k^2,$$ in degree $0$. The cup product composed with the trace, which computes our bilinear form on $\RR p_* \T_{\PP^1}$, gives the Serre duality isomorphism $H^0(\PP^1, \cO_{\PP^1})^* \simeq H^1(\PP^1,\Omega_{\PP^1})$.

From this, we can express the bilinear form of interest in terms of evaluation on $H^0(\PP^1, \cO_{\PP^1})\oplus H^0(\PP^1,\cO_{\PP^1})^* \to k$. Indeed, our form can be written as the composition: 
\[
\big(H^0(\PP^1,\cO_{\PP^1}) \oplus H^0(\PP^1,\cO_{\PP^1})^*\big) \otimes \big(H^0(\PP^1,\cO_{\PP^1}) \oplus H^0(\PP^1,\cO_{\PP^1})^*\big) \to
\]
\[
\big(H^0(\PP^1, \cO_{\PP^1}) \otimes H^0(\PP^1,\cO_{\PP^1})^*\big) \oplus  \big(H^0(\PP^1, \cO_{\PP^1})^* \otimes H^0(\PP^1,\cO_{\PP^1})\big) \to k,
\]
where the first map is projecting onto the cross terms, and the second is given by 
\[
e \otimes f + f' \otimes e' \mapsto f(e) + f'(e').
\]
From this formula, we readily verify that Gram matrix is 
\[
\left(\begin{array}{cc}
0 & 1 \\
1 & 0
\end{array}\right).
\]
	\end{proof}

	\begin{proposition}\label{dP2k3} 
$\chi^{\bA^1}(\PP^{2})=H+\langle 1 \rangle$ in $\GW(k)$.
	\end{proposition} 

	\begin{remark}This provides an alternate proof of the $n=2$ case of \cite[Example 1.7]{hoyois2015quadratic}. 
	\end{remark}

	\begin{proof}
We refer to the spectral sequences depicted in Figures \ref{ssP2} and \ref{ssP2r} for computations of underlying vector spaces. We first note that the rank 2 bilinear form on $H^0(\PP^2,\cO_{\PP^2}) \oplus H^{2}(\PP^2,  \Omega^2_{\PP^2})$ is equal to $H$, by an argument virtually identical to that in Proposition \ref{dP1k2}.

To complete the proposition, we will show that the symmetric bilinear form
\begin{equation}\label{pP2}
\begin{tikzcd}
\Hy^{1}(\PP^{2},T^{*}_{\PP^{2}})\otimes \Hy^{1}(\PP^{2},T^{*}_{\PP^{2}}) \arrow[r,"\cup"] & \Hy^{2}(\PP^{2},\omega_{\PP^{2}}) \arrow[r,"\Tr"]  & k
\end{tikzcd}
\end{equation}
is the map $k\times k \rightarrow k$ given by $(x,y)\mapsto xy$, which is $\langle 1 \rangle$ as defined in Definition \ref{form_a}.

We first find a basis for $\Hy^{1}(\PP^{2},T^{*}_{\PP^{2}})$. To this end, consider the Euler sequence 
\begin{equation*}
\begin{tikzcd}
0 \arrow[r] & T^{*}_{\PP^{2}} \ar[r,"\psi"] & \sh{O}(-1)^{3} \ar[r,"\phi"] & \sh{O}_{\PP^{2}} \ar[r] & 0.
\end{tikzcd}
\end{equation*}
Say $\PP^2=\PP(k[x,y,z])$. Let $f\in k[x,y,z]$, we denote the distinguished open set of $f$ by $\PP^{2}_{f}$. Following \cite[Theorem II.8.13]{HartshorneAG}, the homomorphism $\psi$ is defined on distinguished open sets as follows
\begin{align*}
\psi\left(f_{1}\d{y}{x}+f_{2}\d{z}{x}\right)&=\left(-\frac{y}{x^{2}}f_{1}-\frac{z}{x^{2}}f_{2},\frac{f_{1}}{x},\frac{f_{2}}{x}\right) \text{ on }\PP^{2}_{x},\\ 
\psi\left(g_{1}\d{x}{y}+g_{2}\d{z}{y}\right)&=\left(\frac{g_{1}}{y},-\frac{x}{y^{2}}g_{1}-\frac{z}{y^{2}}g_{2},\frac{g_{2}}{y}\right) \text{ on }\PP^{2}_{y}, \text{ and}\\ 
\psi\left(h_{1}\d{x}{z}+h_{2}\d{y}{z}\right)&=\left(\frac{h_{1}}{z},\frac{h_{2}}{z},-\frac{x}{z^{2}}h_{1}-\frac{y}{z^{2}}h_{2}\right) \text{ on }\PP^{2}_{z}. 
\end{align*}
with $$f_{i}\defeq f_{i}(y/x,z/x),$$ $$g_{i}\defeq g_{i}(x/y,z/y),$$ and $$h_{i}\defeq h_{i}(x/z,y/z)$$ for $i\in\{1,2\}$.  The homomorphism $\phi$ is defined by $$\phi(s_{0},s_{1},s_{2})=xs_{0}+ys_{1}+zs_{2}.$$ 

Let $\delta:\Hy^{0}(\PP^{2},\sh{O}_{\PP^{2}}) \rightarrow \Hy^{1}(\PP^{2},T^{*}_{\PP^{2}})$ be the zeroth connecting homomorphism in the long exact sequence associated to the Euler sequence. Since $\delta$ is an isomorphism, $\delta(1)$ is a basis for $H^{1}(\PP^{2},T^{*}_{\PP^{2}})\cong k$. We next use \v{C}ech cohomology to calculate $\delta(1)$. 

Let $\mathscr{U}$ be the affine covering of $\PP^{2}$ by distinguished open sets $\PP^{2}_{x}$, $\PP^{2}_{y}$ and $\PP^{2}_{z}$. 
In what follows, $\bigl(C^{*}(\mathscr{U},\sh{F}),d^{*}\bigr)$ denotes the \v{C}ech complex associated to $\sh{F}$ and $\mathscr{U}$, where $\sh{F}$ is a sheaf of abelian groups on $\PP^{2}$.

We recall the definition of the connecting homomorphism. The Euler sequence induces a short exact sequence of cochain complexes, \cite[Theorem III.4.5]{HartshorneAG}: 
\[
0\rightarrow C^{*}\bigl(\mathscr{U},T^{*}_{\PP^{2}}\bigr) \rightarrow C^{*}\bigl(\mathscr{U},\sh{O}(-1)^{3}\bigr) \rightarrow C^{*}\bigl(\mathscr{U},\sh{O}_{\PP^{2}}\bigr) \rightarrow 0,
\]
then $\delta$ is defined by diagram chasing in the diagram below, 
\begin{equation}
\begin{tikzcd}\label{cech1}
0 \arrow[r] &C^{0}\bigl(\mathscr{U},T^{*}_{\PP^{2}}\bigr) \arrow[r] \arrow[d,"d^{0}"] & C^{0}\bigl(\mathscr{U},\sh{O}(-1)^{3}\bigr)  \arrow[d,"d^{0}"]  \arrow[r] & C^{0}\bigl(\mathscr{U},\sh{O}_{\PP^{2}}\bigr) \arrow[d,"d^{0}"] \arrow[r] & 0\\
0 \arrow[r] & C^{1}\bigl(\mathscr{U},T^{*}_{\PP^{2}}\bigr) \arrow[r] & C^{1}\bigl(\mathscr{U},\sh{O}(-1)^{3}\bigr)  \arrow[r]  & C^{1}\bigl(\mathscr{U},\sh{O}_{\PP^{2}}\bigr)  \arrow[r] & 0.
\end{tikzcd}
\end{equation}
Take $(1,1,1) \in \sh{O}_{\PP^{2}}(\PP^{2}_{x})\times \sh{O}_{\PP^{2}}(\PP^{2}_{y}) \times \sh{O}_{\PP^{2}}(\PP^{2}_{z})$ the generator of $H^{0}(\PP^{2},\sh{O}_{\PP^{2}})\cong k$. By the definition of $\phi$, $d^{0}$ and $\psi$ we see that $\delta(1,1,1)=\left(\tfrac{y}{x}d\bigl(\tfrac{x}{y}\bigr),\tfrac{z}{y}d\bigl(\tfrac{y}{z}\bigr),\tfrac{z}{x}d\bigl(\tfrac{x}{z}\bigr)\right) \in H^{1}(\PP^{2},T^{*}_{\PP^{2}})$, as diagram \eqref{dc1} illustrates. 
\begin{equation}
\begin{tikzcd}\label{dc1}
& \left(\left(\tfrac{1}{x},0,0\right),\left(0,\tfrac{1}{y},0\right),\left(0,0,\tfrac{1}{z}\right)\right)  \arrow[d,mapsto]  \arrow[r,mapsto] & (1,1,1)\\
\left(\tfrac{y}{x}d\bigl(\tfrac{x}{y}\bigr),\tfrac{z}{y}d\bigl(\tfrac{y}{z}\bigr),\tfrac{z}{x}d\bigl(\tfrac{x}{z}\bigr)\right) \arrow[r,mapsto] & \left(\bigl(\tfrac{1}{x},-\tfrac{1}{y},0\bigr),\bigl(0,\tfrac{1}{y},-\tfrac{1}{z}\bigr),\bigl(\tfrac{1}{x},0,-\tfrac{1}{z}\bigr)\right) &
\end{tikzcd}
\end{equation}

Our next task is to calculte $Q(\delta(1)\otimes\delta(1))$. We first compute $\delta(1)\cup\delta(1)$, as $Q(\delta(1)\otimes\delta(1))=\Tr(\delta(1)\cup\delta(1))$. Consider the diagram below
\begin{equation}\label{cup}
\begin{tikzcd}[column sep=1.45em]
\Hy^{0}(\PP^{2},\sh{O}_{\PP^{2}})\otimes \Hy^{1}(\PP^{2},T^{*}_{\PP^{2}}) \arrow[r,"\delta\otimes id"] \arrow[d,"\cup"] & \Hy^{1}(\PP^{2},T^{*}_{\PP^{2}})\otimes\Hy^{1}(\PP^{2},T^{*}_{\PP^{2}}) \arrow[d,"\cup"] & & \\
\Hy^{1}(\PP^{2}, \sh{O}_{\PP^{2}} \otimes T^{*}_{\PP^{2}}) \arrow[r,"\partial"]  & \Hy^{2}(\PP^{2},T^{*}_{\PP^{2}}\otimes T^{*}_{\PP^{2}}) \arrow[r,"\wedge"]   & \Hy^{2}(\PP^{2},\omega_{\PP^{2}}) \arrow[r,"\Tr"]  & k,
\end{tikzcd}
\end{equation}
where $\partial$ is the first connecting homomorphism in the long exact sequence associated to the short exact sequence
\begin{equation*}
\begin{tikzcd}
0 \arrow[r] & T^{*}_{\PP^{2}}\otimes T^{*}_{\PP^{2}} \ar[r,"\psi\otimes id"] & \sh{O}(-1)^{3}\otimes T^{*}_{\PP^{2}} \ar[r,"\phi\otimes id"] & \sh{O}_{\PP^{2}}\otimes T^{*}_{\PP^{2}} \ar[r] & 0,
\end{tikzcd}
\end{equation*}
which is obtained by tensorizing the Euler sequence with $T^{*}_{\PP^{2}}$. By commutativity of diagram \eqref{cup}, we have that $\delta(1)\cup\delta(1)=\partial(1 \cup \delta(1)) = \partial(\delta(1))$.

\begin{claim}\label{cocycle}
Let $\alpha$ denote the \v{C}ech cocycle 
\[
\dfrac{z^{2}}{xy}\d{x}{z}\wedge \d{y}{z} \in \Hy^{2}(\PP^{2},\omega_{\PP^{2}}).
\]
We claim that the composite $%
\begin{tikzcd}
\Hy^{1}(\PP^{2},\sh{O}_{\PP^{2}} \otimes T^{*}_{\PP^{2}}) \arrow[r,"\partial"]  & \Hy^{2}(\PP^{2},T^{*}_{\PP^{2}}\otimes T^{*}_{\PP^{2}}) \arrow[r,"\wedge"]   & \Hy^{2}(\PP^{2},\omega_{\PP^{2}})
\end{tikzcd}$ satisfies $$\delta(1) \mapsto -\alpha.$$
\end{claim}
Given the claim, we complete the argument. The canonical trace map sends the form of \cite[Remark 7.1.1]{HartshorneAG} to $1$. Since $\alpha$ differs from this by the permutation swapping $x$ and $z$, which has sign $-1$, we deduce that $\alpha$ maps to $-1$ under the trace. By claim \ref{cocycle} it follows that $\Tr(\delta(1)\tensor\delta(1))=1$. 
	\end{proof}

	\begin{proof}[Proof of claim \ref{cocycle}]
We calculate $\partial(\delta(1))$ via the commutative diagram below.
\begin{equation}
\begin{tikzcd}[column sep=small]\label{cech2}
0 \arrow[r] &C^{1}\bigl(\mathscr{U},T^{*}_{\PP^{2}}\otimes T^{*}_{\PP^{2}} \bigr) \arrow[r] \arrow[d,"d^{1}"] & C^{1}\bigl(\mathscr{U},\sh{O}(-1)^{3}\otimes T^{*}_{\PP^{2}}\bigr)  \arrow[d,"d^{1}"]  \arrow[r] & C^{1}\bigl(\mathscr{U},\sh{O}_{\PP^{2}}\otimes T^{*}_{\PP^{2}}\bigr) \arrow[d,"d^{1}"] \arrow[r] & 0\\
0 \arrow[r] & C^{2}\bigl(\mathscr{U},T^{*}_{\PP^{2}}\otimes T^{*}_{\PP^{2}}\bigr) \arrow[r] & C^{2}\bigl(\mathscr{U},\sh{O}(-1)^{3}\otimes T^{*}_{\PP^{2}}\bigr)  \arrow[r]  & C^{2}\bigl(\mathscr{U},\sh{O}_{\PP^{2}}\otimes T^{*}_{\PP^{2}}\bigr)  \arrow[r] & 0.
\end{tikzcd}
\end{equation}
Note that $$\delta(1) = \left(\tfrac{y}{x}\otimes d\bigl(\tfrac{x}{y}\bigr),\tfrac{z}{y}\otimes d\bigl(\tfrac{y}{z}\bigr),\tfrac{z}{x}\otimes d\bigl(\tfrac{x}{z}\bigr)\right) \in H^{1}(\PP^{2},\sh{O}_{\PP^{2}}\otimes T^{*}_{\PP^{2}}).$$ For ease of notation, we define
$$ A \defeq
\left(\Bigl(\frac{y}{x^{2}}\otimes\d{x}{y},0,0\Bigr),\Bigl(0,\frac{z}{y^{2}}\otimes\d{y}{z},0\Bigr),\Bigl(\frac{z}{x^{2}}\otimes\d{x}{z},0,0,\Bigr)\right)$$ and $$ B \defeq \frac{z^{2}}{xy}\d{x}{z}\otimes \d{y}{z}.$$ By definition of $\phi$ we have that $(\phi\otimes id)(A)=\delta(1)$. Moreover, $d^{1}(A)$ is equal to 
\begin{align}\label{dA}
\left(1\otimes\Bigl(\frac{y}{x^{2}}\d{x}{y}-\frac{z}{x^{2}}\d{x}{z}\Bigr),1\otimes\frac{z}{y^{2}}\d{y}{z},0\right) \in \left(\sh{O}(-1)^{3}\otimes T^{*}_{\PP^{2}}\right)(\PP^{2}_{xyz}).
\end{align} To simplify \eqref{dA} we use the equality
\begin{align}\label{diff}
\frac{z}{x}\d{x}{z}=\frac{z}{x}\left(\frac{y}{z}\d{x}{y}+\frac{x}{y}\d{y}{z}\right)=\frac{y}{x}\d{x}{y}+\frac{z}{y}\d{y}{z}.
\end{align}
Substituing \eqref{diff} into the first coordinate of \eqref{dA} yields
\begin{align}\label{dAs}
d^{1}(A)=\left(-\frac{z}{xy},\frac{z}{y^{2}},0\right)\otimes\d{y}{z}.
\end{align}

We now compute $(\psi\otimes id)(B)$ and $d^{1}(A)+(\psi\otimes id)(B)$. By definition of $\psi$ we have
\begin{align}\label{psiB}
(\psi\otimes id)(B)=\left(\frac{z^{2}}{xy}\frac{1}{z},0,-\frac{z^{2}}{xy}\frac{x}{z^{2}}\right)\otimes \d{y}{z}=\left(\frac{z}{xy},0,-\frac{1}{y}\right)\otimes \d{y}{z},
\end{align}
and combining \eqref{dAs} and \eqref{psiB} we obtain
\begin{align*}
d^{1}(A)+(\psi\otimes id)(B)=\left(0,\frac{z}{y^{2}},-\frac{1}{y}\right)\otimes \d{y}{z}=\psi\otimes id\left(-\d{z}{y}\otimes\d{y}{z}\right).
\end{align*}
Note that we have actually proved that $\partial(\delta(1))=\d{z}{y}\otimes\d{y}{z}-B$, as we can see in the diagram below.
\begin{equation}
\begin{tikzcd}
& A  \arrow[d,mapsto]  \arrow[r,mapsto] & \delta(1)\\
d\bigl(\tfrac{z}{y}\bigr)\otimes d\bigl(\tfrac{y}{z}\bigr)-B  \arrow[r,mapsto] & d^{1}(A) &
\end{tikzcd}
\end{equation}

Since
\begin{align*}
\d{z}{y}\otimes\d{y}{z}=\d{z}{y}\otimes\d{y}{z}&=-\frac{y^{2}}{z^{2}}\d{y}{z}\otimes\d{y}{z},
\end{align*}
it follows that $d\bigl(\tfrac{z}{y}\bigr)\wedge d\bigl(\tfrac{y}{z}\bigr)=0$. From this we obtain the claim:
\begin{equation*}
\begin{tikzcd}[row sep=tiny,column sep=scriptsize]
\Hy^{1}(\PP^{2},T^{*}_{\PP^{2}}\otimes\sh{O}_{\PP^{2}}) \arrow[r,"\partial"]  & \Hy^{2}(\PP^{2},T^{*}_{\PP^{2}}\otimes T^{*}_{\PP^{2}}) \arrow[r,"\wedge"]   & \Hy^{2}(\PP^{2},\omega_{\PP^{2}})\\
\delta(1) \ar[r,mapsto] & \d{z}{y}\otimes\d{y}{z}-B \ar[r,mapsto] & -\alpha.
\end{tikzcd}
\end{equation*}
	\end{proof}

	\begin{proposition}\label{dBl}
$\chi^{\bA^1}(\Bl_{0}\PP^{2})=2H$ in $\GW(k)$.
	\end{proposition}
	\begin{proof}
As before, we first consider the pairing on $H^0(\Bl_{0}\PP^2,\cO_{\Bl_{0}\PP^2}) \oplus  H^2(\Bl_{0}\PP^2, \Omega^2_{\Bl_{0}\PP^2})$ and show it is $H$ in $\GW(k)$. The argument here is again as in Proposition \ref{dP1k2}.

Now let $Q'$ denote the class of the symmetric bilinear form
\begin{equation*}
\begin{tikzcd}
\Hy^{1}(\Bl_{0}\PP^{2},T^{*}_{\Bl_{0}\PP^{2}})\otimes \Hy^{1}(\Bl_{0}\PP^{2},T^{*}_{\Bl_{0}\PP^{2}}) \arrow[r,"\cup"] & \Hy^{2}(\Bl_{0}\PP^{2},\omega_{\Bl_{0}\PP^{2}}) \arrow[r,"\Tr"]  & k.
\end{tikzcd}
\end{equation*}
We will show that $Q'$ isomorphic to $H$. 

To do this it is enough to find a non-zero element $v$ in $H^1(\Bl_0\PP^2, T^{*}_{\Bl_{0}\PP^{2}})$ such that the image of $v \cup v$ in $H^1(\Bl_0\PP^2, \omega_{\Bl_{0}\PP^{2}})$ is $0$. To see this, note that we may extend $v$ to a basis $\{v,w\}$ of $H^1(\Bl_0\PP^2, T^{*}_{\Bl_{0}\PP^{2}})$. Replacing $w$ by $w - \frac{Q'(w,w)}{2Q'(v,w)} v$ (note that we use that the characteristic is not $2$ here), we obtain a new basis $\{ w',v\}$ so that the Gram matrix of $Q'$ is 
\begin{equation*}
\left(\begin{array}{cc}
0 & Q'(w',v)  \\
Q'(w',v) & 0 
\end{array}\right).
\end{equation*} Rescaling $v$, we have 
\begin{equation*}
\left(\begin{array}{cc}
0 & 1  \\
1 & 0 
\end{array}\right).
\end{equation*} which is $H$ in $\GW(k)$. 

The blow-up $T^{*}_{\Bl_{0}\PP^{2}}$ can be described as the projectivization of total space of the bundle $\calO(-1) \oplus \calO$ on $\PP^1$:
\[
T^{*}_{\Bl_{0}\PP^{2}} \cong \Proj_{\Proj k[x,y,z]} k[x,y,z][S,T]/\langle Sx -Ty \rangle \cong \PP_{\Proj k[S,T]} (\calO(-1) \oplus \calO).
\]
Let $\pi: \Bl_{0}\PP^{2} \to \PP^1$ denote the projection. The map $\pi$ induces a map $H^1(\PP^1, \Omega_{\PP^1}) \to H^1(\Bl_{0}\PP^{2}, \pi^* \Omega_{\PP^1})$. Composing with the map $ \pi^* \Omega_{\PP^1} \to \Omega_{\Bl_{0}\PP^{2}}$, we obtain 
\[
\pi^*: H^1(\PP^1, \Omega_{\PP^1}) \to H^1(\Bl_{0}\PP^{2}, \Omega_{\Bl_{0}\PP^{2}}).
\]
Since $H^2(\PP^2,\Omega_{\PP^1} \otimes \Omega_{\PP^1} ) =0$, any $v$ in the image of $\pi^*$ satisfies $Q'(v,v) = 0$ by naturally of the cup product. 

We have thus reduced the problem to showing that $\pi^*$ is nonzero. Since $\pi i'$ is the identity on $\PP^1$, the composition of $\pi^*$ with the map $$ \Hy^{1}\bigl(\Bl_{0}\PP^{2};T^{*}_{\Bl_{0}\PP^{2}}\bigr) \to \Hy^{1}\bigl(\Bl_{0}\PP^{2};\Omega_{\Bl_{0}\PP^{2}/\PP^{2}}\bigr) \cong \Hy^{1}\bigl(\Bl_{0}\PP^{2};i'_* \Omega_{\PP^1}\bigr) \cong H^1(\PP^1, \Omega_{\PP^1})$$ of Equation \eqref{les1} is the identity. Since $ H^1(\PP^1; \Omega_{\PP^1}) \cong k$, it follows that $\pi^*$ is nonzero as claimed.
	\end{proof}

Now we can verify the equality \eqref{pbus}; that is, that $\chi^{\bA^1}(\PP^{1})+\chi^{\bA^1}(\PP^{2})=\chi^{\bA^1}(\Spec k)+\chi^{\bA^1}(\Bl_{0}\PP^{2}).$ Substituting in the results of Propositions \ref{dspeck}, \ref{dP1k2}, \ref{dP2k3} and \ref{dBl}, we obtain the result. 

\bibliographystyle{alpha}
\bibliography{ABOWZ}

\begin{thebibliography}{{Sta}20}

\bibitem[AI08]{avramov_iyengar}
Luchezar~L. Avramov and Srikanth~B. Iyengar.
\newblock Gorenstein algebras and {H}ochschild cohomology.
\newblock volume~57, pages 17--35. 2008.
\newblock Special volume in honor of Melvin Hochster.

\bibitem[AJL14]{atjll}
Leovigildo Alonso, Ana Jerem\'{\i}as, and Joseph Lipman.
\newblock Bivariance, {G}rothendieck duality and {H}ochschild homology, {II}:
  {T}he fundamental class of a flat scheme-map.
\newblock {\em Adv. Math.}, 257:365--461, 2014.

\bibitem[AV20]{antieau_vezzosi}
Benjamin Antieau and Gabriele Vezzosi.
\newblock A remark on the {H}ochschild-{K}ostant-{R}osenberg theorem in
  characteristic {$p$}.
\newblock {\em Ann. Sc. Norm. Super. Pisa Cl. Sci. (5)}, 20(3):1135--1145,
  2020.

\bibitem[Bit04]{bittner04}
Franziska Bittner.
\newblock The universal {E}uler characteristic for varieties of characteristic
  zero.
\newblock {\em Compos. Math.}, 140(4):1011--1032, 2004.

\bibitem[BM00]{BargeMorel}
Jean Barge and Fabien Morel.
\newblock Groupe de {C}how des cycles orient\'es et classe d'{E}uler des
  fibr\'es vectoriels.
\newblock {\em C. R. Acad. Sci. Paris S\'er. I Math.}, 330(4):287--290, 2000.

\bibitem[BW]{bachmann_wickelgren}
T.~Bachmann and K.~Wickelgren.
\newblock $\mathbb{A}^1$-{E}uler classes: six functors formalisms, dualities,
  integrality, and linear subspaces of complete intersections.
\newblock Preprint available at \url{https://arxiv.org/abs/2002.01848}.

\bibitem[Cam19]{campbell}
Jonathan~A. Campbell.
\newblock The {$K$}-theory spectrum of varieties.
\newblock {\em Trans. Amer. Math. Soc.}, 371(11):7845--7884, 2019.

\bibitem[CH09]{calmes2009tensor}
Baptiste Calm{\`e}s and Jens Hornbostel.
\newblock Tensor-triangulated categories and dualities.
\newblock {\em Theory Appl. Categ}, 22(6):136--200, 2009.

\bibitem[Don83]{Donaldson83}
S.~K. Donaldson.
\newblock An application of gauge theory to four-dimensional topology.
\newblock {\em J. Differential Geom.}, 18(2):279--315, 1983.

\bibitem[Fas08]{FaselGroupesCW}
Jean Fasel.
\newblock Groupes de {C}how-{W}itt.
\newblock {\em M\'em. Soc. Math. Fr. (N.S.)}, (113):viii+197, 2008.

\bibitem[Fre82]{Freedman82}
Michael~Hartley Freedman.
\newblock The topology of four-dimensional manifolds.
\newblock {\em J. Differential Geometry}, 17(3):357--453, 1982.

\bibitem[Gro61]{egaIII}
Alexander Grothendieck.
\newblock \'el\'ements de g\'eom\'etrie alg\'ebrique : {III}. \'etude
  cohomologique des faisceaux coh\'erents, premi\`ere partie.
\newblock {\em Publications Math\'ematiques de l'IH\'ES}, 11:5--167, 1961.

\bibitem[Har66]{HartshorneRD}
Robin Hartshorne.
\newblock {\em Residues and duality}.
\newblock Lecture notes of a seminar on the work of A. Grothendieck, given at
  Harvard 1963/64. With an appendix by P. Deligne. Lecture Notes in
  Mathematics, No. 20. Springer-Verlag, Berlin-New York, 1966.

\bibitem[Har13]{HartshorneAG}
R.~Hartshorne.
\newblock {\em Algebraic Geometry}.
\newblock Graduate Texts in Mathematics. Springer New York, 2013.

\bibitem[HKR62]{HKR62}
G.~Hochschild, Bertram Kostant, and Alex Rosenberg.
\newblock Differential forms on regular affine algebras.
\newblock {\em Trans. Amer. Math. Soc.}, 102:383--408, 1962.

\bibitem[Hoy15]{hoyois2015quadratic}
Marc Hoyois.
\newblock A quadratic refinement of the grothendieck--lefschetz--verdier trace
  formula.
\newblock {\em Algebraic \& Geometric Topology}, 14(6):3603--3658, 2015.

\bibitem[Hoy17]{hoyois-equivariant}
Marc Hoyois.
\newblock The six operations in equivariant motivic homotopy theory.
\newblock {\em Advances in Mathematics}, 305:197--279, 2017.

\bibitem[Hu05]{Hu_Picard}
Po~Hu.
\newblock On the {P}icard group of the stable {$\Bbb A^1$}-homotopy category.
\newblock {\em Topology}, 44(3):609--640, 2005.

\bibitem[KW17]{CubicSurface}
Jesse Kass and Kirsten Wickelgren.
\newblock An arithmetic count of the lines on a smooth cubic surface.
\newblock Accepted for publication in Compositio Mathematica, preprint
  available at \url{https://arxiv.org/abs/1708.01175}, 2017.

\bibitem[Lev20]{Levine-EC}
Marc Levine.
\newblock Aspects of enumerative geometry with quadratic forms.
\newblock {\em Doc. Math.}, 25:2179--2239, 2020.

\bibitem[Lip87]{lipman1987residues}
J.~Lipman.
\newblock {\em Residues and Traces of Differential Forms via Hochschild
  Homology}.
\newblock Contemporary mathematics - American Mathematical Society. American
  Mathematical Society, 1987.

\bibitem[LR20]{levine_raksit}
Marc Levine and Arpon Raksit.
\newblock Motivic {G}auss-{B}onnet formulas.
\newblock {\em Algebra Number Theory}, 14(7):1801--1851, 2020.

\bibitem[McC01]{mccleary01}
John McCleary.
\newblock {\em A user's guide to spectral sequences}, volume~58 of {\em
  Cambridge Studies in Advanced Mathematics}.
\newblock Cambridge University Press, Cambridge, second edition, 2001.

\bibitem[Mor06]{Mor06}
Fabien Morel.
\newblock {$\Bbb A^1$}-algebraic topology.
\newblock In {\em International {C}ongress of {M}athematicians. {V}ol. {II}},
  pages 1035--1059. Eur. Math. Soc., Z\"{u}rich, 2006.

\bibitem[Mor12]{A1-alg-top}
Fabien Morel.
\newblock {\em $\mathbb{A}^1$-Algebraic Topology over a Field}.
\newblock Lecture Notes in Mathematics. Springer Berlin Heidelberg, 2012.

\bibitem[MV99]{morelvoevodsky1998}
F.~Morel and V.~Voevodsky.
\newblock {${\mathbb A}^1$}-homotopy theory of schemes.
\newblock {\em Inst. Hautes \'Etudes Sci. Publ. Math.}, (90):45--143 (2001),
  1999.

\bibitem[Nee18]{Neeman_GDHH}
Amnon Neeman.
\newblock The relation between {G}rothendieck duality and {H}ochschild
  homology.
\newblock In {\em {$K$}-{T}heory---{P}roceedings of the {I}nternational
  {C}olloquium, {M}umbai, 2016}, pages 91--126. Hindustan Book Agency, New
  Delhi, 2018.

\bibitem[Rio05]{riou2005dualite}
Jo{\"e}l Riou.
\newblock Dualit{\'e} de spanier--whitehead en g{\'e}om{\'e}trie
  alg{\'e}brique.
\newblock {\em Comptes Rendus Mathematique}, 340(6):431--436, 2005.

\bibitem[Roh52]{Rohlin52}
V.~A. Rohlin.
\newblock New results in the theory of four-dimensional manifolds.
\newblock {\em Doklady Akad. Nauk SSSR (N.S.)}, 84:221--224, 1952.

\bibitem[R{\"o}n]{rondigs}
Oliver R{\"o}ndigs.
\newblock The {G}rothendieck ring of varieties and algebraic k-theory of
  spaces.
\newblock arXiv:1611.09327.

\bibitem[Sch10]{schlichting10}
Marco Schlichting.
\newblock Hermitian {$K$}-theory of exact categories.
\newblock {\em J. K-Theory}, 5(1):105--165, 2010.

\bibitem[Sch17]{Schlichting17}
Marco Schlichting.
\newblock Hermitian {$K$}-theory, derived equivalences and {K}aroubi's
  fundamental theorem.
\newblock {\em J. Pure Appl. Algebra}, 221(7):1729--1844, 2017.

\bibitem[{Sta}20]{stacksproject}
The {Stacks project authors}.
\newblock The stacks project.
\newblock \url{https://stacks.math.columbia.edu}, 2020.

\bibitem[TV11]{Toen_Vezzosi11}
Bertrand To\"{e}n and Gabriele Vezzosi.
\newblock Alg\`ebres simpliciales {$S^1$}-\'{e}quivariantes, th\'{e}orie de de
  {R}ham et th\'{e}or\`emes {HKR} multiplicatifs.
\newblock {\em Compos. Math.}, 147(6):1979--2000, 2011.

\bibitem[Voe03]{voevodsky2003motivic}
Vladimir Voevodsky.
\newblock Motivic cohomology with $\mathbf{Z}/2$-coefficients.
\newblock {\em Publications Math{\'e}matiques de l'Institut des Hautes
  {\'E}tudes Scientifiques}, 98:59--104, 2003.

\bibitem[Zak17]{Z-kth-ass}
Inna Zakharevich.
\newblock The {$K$}-theory of assemblers.
\newblock {\em Adv. Math.}, 304:1176--1218, 2017.

\end{thebibliography}
\end{document}